\documentclass[11pt,authoryear]{elsarticle}
\usepackage{fullpage,comment}

\usepackage{amssymb}
\usepackage{amsthm}

\usepackage{lineno}
\usepackage{placeins, microtype} % this only for \FloatBarrier
\usepackage{amsmath,color}
\usepackage{cases}
\usepackage[hidelinks]{hyperref}
\usepackage{mathtools}

\newtheorem{thm}{Theorem}
\newtheorem{lemma}{Lemma}
\newtheorem{remark}{Remark}
\newtheorem{prop}{Proposition}

\def\R{\mathbb{R}}
\def\Z{\mathbb{Z}}
\def\eqd{\,{\buildrel d \over =}\,}
\def\P{{\mathbb P}}     
\def\E{{\mathbb E}}     
\newcommand{\toL}{\,{\buildrel {d} \over \longrightarrow}\,}

\newcommand{\coleq}{\mathrel{\mathop:}=}
\newcommand{\eqcol}{=\mathrel{\mathop:}}
\newcommand*\mathinhead[2]{\texorpdfstring{$#1$}{#2}}

\journal{arXiv}

\begin{document}

\begin{frontmatter}



\title{Latent mutations in the ancestries of alleles under selection}


\author[1,2]{Wai-Tong (Louis) Fan}
\ead{waifan@iu.edu}
\author[2]{John Wakeley}
\ead{wakeley@fas.harvard.edu}

\affiliation[1]{
organization={Department of Mathematics, Indiana University},
addressline={831 East 3rd St}, 
city={Bloomington},
postcode={47405}, 
state={IN},
country={USA}
}

\affiliation[2]{
organization={Department of Organismic and Evolutionary Biology, Harvard University},
addressline={16 Divinity Ave}, 
city={Cambridge},
postcode={02138}, 
state={MA},
country={USA}
}

\begin{abstract}
We consider a single genetic locus with two alleles $A_1$ and $A_2$ in a large haploid population. The locus is subject to selection and two-way, or recurrent, mutation. Assuming the allele frequencies follow a Wright-Fisher diffusion and have reached stationarity, we describe the asymptotic behaviors of the conditional gene genealogy and the latent mutations of a sample with known allele counts, when the count $n_1$ of allele $A_1$ is fixed, and when either or both the sample size $n$ and the selection strength $\lvert\alpha\rvert$ tend to infinity.  Our study extends previous work under neutrality to the case of non-neutral rare alleles, asserting that when selection is not too strong relative to the sample size, even if it is strongly positive or strongly negative in the usual sense ($\alpha\to -\infty$ or $\alpha\to +\infty$), the number of latent mutations of the $n_1$ copies of allele $A_1$ follows the same distribution as the number of alleles in the Ewens sampling formula.  On the other hand, very strong positive selection relative to the sample size leads to neutral gene genealogies with a single ancient latent mutation.  We also demonstrate robustness of our asymptotic results against changing population sizes, when one of $\lvert\alpha\rvert$ or $n$ is large.  
\end{abstract}

%%%%\begin{graphicalabstract}
%%%%\end{graphicalabstract}
%%%%\begin{highlights}
%%%%\item Research highlight 1
%%%%\item Research highlight 2
%%%%\end{highlights}

\begin{keyword}
Recurrent mutation \sep selection \sep Ewens sampling formula \sep coalescent \sep Wright-Fisher diffusion 



\end{keyword}

\end{frontmatter}

%\linenumbers

\section{Introduction} \label{sec:intro}

The observed copies of a particular allele in a sample descend from an unknown number of distinct mutations.  If $k_1$ is the number of these `latent' mutations for allele $A_1$ when it is observed $n_1$ times in a sample, then $k_1\in\{1,2,\ldots,n_1\}$.  Although latent mutations are not observed directly, they can be modeled as outcomes of the stochastic ancestral process of a sample and inferred from patterns of variation in DNA data \citep{HarpakEtAl2016,SeplyarskiyEtAl2021,JohnsonEtAl2022}.  Analytical results on the distribution and timing of latent mutations of rare neutral alleles are given in \citet{WakeleyEtAl2023}.  Here we consider non-neutral alleles which may be under strong selection and which may or may not be rare.  We take two different approaches to modeling latent mutations under selection and recurrent mutation.  The first approach uses the idea of coalescence in a random background of allele frequencies in the population \citep{BartonEtAl2004}.  The second uses the conditional ancestral selection graph \citep{Slade2000a} and demonstrates results consistent with those from the first approach.

\citet{WakeleyEtAl2023} also contains an application to the frequencies of single-nucleotide sites with counts $n_1\in\{1,2,\ldots,40\}$ of synonymous mutations in a subsample of $57$K non-Finnish European individuals ($n=114$K) from the \textit{gnomAD} database \citep{KarczewskiEtAl2020}.  Dramatic differences in sample frequency distributions of rare alleles with different mutation rates, categorized by the `Roulette' method of \citet{SeplyarskiyEtAl2023}, were well explained by an empirical demographic model with recurrent mutation but no selection.  \citet[Fig.~3a]{SeplyarskiyEtAl2023} showed using simulations that a neutral, parametric demographic model fitted to these data also explained the frequencies of mutation in counts $n_1 \leq 10^4$.  Polymorphic sites with small mutation counts comprise the bulk of variation in humans.  They represent a rich source of information about demographic history and possibly selection.  Sites with $n_1\in\{1,2,\ldots,40\}$ make up about $95$\% of all polymorphic sites in the \textit{gnomAD} data used in \citet{WakeleyEtAl2023}. 

At present humans are the only species with sufficient genomic data to apply such models of rare variants which rely on limiting approximations for large sample sizes.  Whereas the neutral models in \citet{WakeleyEtAl2023} and \citet{SeplyarskiyEtAl2023} also account for the extreme population growth of humans \citep{KeinanAndClark2012,GazaveEtAl2014,GaoAndKeinan2016}, in considering selection here we focus on populations of constant size.  Previous theoretical work on populations of constant size has shown that distributions of rare alleles are in fact unaffected even by moderately strong selection \citep{JoyceAndTavare1995,Joyce1995}.  Specifically, the counts of latent mutations obey the independent Poisson statistics of rare alleles in the Ewens sampling formula \citep{Ewens1972,ArratiaBarbourAndTavare1992,ArratiaBarbourAndTavare2003}.  This is also the case in \citet{WakeleyEtAl2023} when the population size is constant.  In the present work we investigate the robustness of these results to very strong selection.  Theory also predicts that rare alleles tend to be young \citep{KimuraAndOhta1973,Watterson1976}.  \citet{MathiesonAndMcVean2014} and \citet{PlattEtAl2019} have demonstrated empirically that rare non-synonymous or otherwise functional alleles in the human genome are even younger than non-functional rare alleles.  In the present work we also investigate how strong selection and rarity affect the ages of latent mutations. 

We assume there are two possible alleles, $A_1$ and $A_2$, at a single genetic locus in a large haploid population. We begin by assuming that the population size $N$ is constant over time.  In Section~\ref{S:varying} we consider time-varying population size.  One allele or the other is favored by directional selection.  Mutation is recurrent and happens in both directions.  In the diffusion approximation, time is measured in proportion to $N_e$ generations where $N_e$ is the effective population size \citep{Ewens2004}.  Under the Wright-Fisher model of reproduction, $N_e=N$.  Under the Moran model of reproduction \citep{Moran1958,Moran1962}, $N_e=N/2$.  With these assumptions, the frequency of $A_1$ alleles is well approximated by a process $X$ that solves \eqref{eq:sde} below and has parameters $\theta_1$, $\theta_2$ and $\alpha$ as $N\to\infty$.  For a haploid population, $\theta_i=2N_e u_i$ and $\alpha=2N_e s$, in which $u_i$ is the per-generation rate of $A_{3-i} \to A_i$ mutations and $s$ is the selection coefficient.  If there is no dominance, these results can be extended to diploids, in which case $\theta_i=4N_e u_i$ and $\alpha=4N_e s$.

Thus, we assume that allele-frequency dynamics in the population obey the Wright-Fisher diffusion \citep{Fisher1930b,Wright1931,Ewens2004} with parameters $\theta_1$ and $\theta_2$ for mutations $A_2 \to A_1$ and $A_1 \to A_2$, respectively, and $\alpha$ for the selective advantage (if $\alpha>0$) or disadvantage (if $\alpha<0$) of allele $A_1$.   
That is, we let $X(t)$ be the relative frequency of $A_1$ in the population at time $t$, and assume that   its forward-time dynamics is described by the stochastic differential equation  
\begin{linenomath*}
\begin{equation}
dX(t) = \left[\frac{\theta_1}{2} (1-X(t)) - \frac{\theta_2}{2} X(t) + \frac{\alpha}{2} X(t) (1-X(t))\right]dt + \sqrt{X(t) (1-X(t))}\, dW_t, \quad t>0 \label{eq:sde}
\end{equation}
\end{linenomath*} 
in which $W_t$ is the Wiener process, also called the standard Brownian motion.  

Both of the approaches (random background and ancestral selection graph) we take to modeling latent mutations  rely on the assumption that the population has reached equilibrium, which occurs in the limit $t\to\infty$.  The stationary probability density of $X$ is    
\begin{linenomath*}
\begin{equation} 
\phi_\alpha(x) = C x^{\theta_1 - 1} (1-x)^{\theta_2 - 1} e^{\alpha x}, \quad 0<x<1 \label{eq:phix}
\end{equation}
\end{linenomath*}
\citep{Wright1931,Ewens2004}.  We explicitly denote the dependence on $\alpha$ because this parameter plays a key role in what follows.  The normalizing constant $C$ guarantees that $\int_0^1 \phi_\alpha(x)dx=1$.  It is given by  
\begin{linenomath*}
\begin{equation} 
C = \frac{\Gamma(\theta_1+\theta_2)}{\Gamma(\theta_1) \Gamma(\theta_2) {}_1F_1(\theta_1;\theta_1+\theta_2;\alpha)} \label{eq:C}
\end{equation}
\end{linenomath*}
in which $\Gamma(a)$ is the gamma function and $_{1}F_{1}(a;b;z)$ is the confluent hypergeometric function, or Kummer's function; see \citet{AbramowitzAndStegun1964} and \citet{Slater1960}.

By definition, latent mutations occur in the ancestry of a sample.  When a sample of total size $n$ is taken from a population with stationary density \eqref{eq:phix}, it will contain a random number ${\mathcal N}_1$ of copies of allele $A_1$ and ${\mathcal N}_2 = n-{\mathcal N}_1$ copies of allele $A_2$.  The probability that ${\mathcal N}_1$ is equal to $n_1$ is equal to 
\begin{linenomath*}
\begin{align} 
q(n_1,n_2) &\coleq \P({\mathcal N}_1=n_1;n,\alpha,\theta_1,\theta_2) \notag \\[4pt]
 &= \int_0^1 \binom{n}{n_1} x^{n_1} (1-x)^{n-n_1} \phi_\alpha(x) dx  \notag \\[4pt]
&= C \binom{n_1+n_2}{n_1} \frac{\Gamma(\theta_1+n_1)\Gamma(\theta_2+n_2)}{\Gamma(\theta_1+\theta_2+n_1+n_2)} {}_1F_1(\theta_1 + n_1;\theta_1+\theta_2+n_1+n_2;\alpha) \label{eq:qn1}
\end{align}
\end{linenomath*}
for $n_1 \in \{0,1,\ldots,n\}$ and $n_2=n-n_1$, and with $C$ again given by \eqref{eq:C}.  The notation $q(\cdot)$ is from \citet{Slade2000a,Slade2000b} and is convenient for the ancestral selection graph.

Suppose now we are given the sample count, that is, we know that among the $n$ uniformly sampled haploid individuals,  $n_1$ of them are of type 1 and the remaining $n_2=n-n_1$ are of type 2. Then the posterior density of the population frequency of $A_1$ conditional on the sample is
\begin{linenomath*}
\begin{align} 
\phi_\alpha^{(n_1,n_2)}(x) &= \frac{\binom{n_1+n_2}{n_1} x^{n_1} (1-x)^{n_2} \phi_\alpha(x)}{q(n_1,n_2)} \label{eq:condphi1} \\[4pt]
&= \frac{\Gamma(\theta_1+\theta_2+n_1+n_2) x^{\theta_1+n_1-1} (1-x)^{\theta_2+n_2-1} e^{\alpha x}}{\Gamma(\theta_1+n_1)\Gamma(\theta_2+n_2){}_1F_1(\theta_1 + n_1;\theta_1+\theta_2+n_1+n_2;\alpha)} \label{eq:condphi2}
\end{align}
\end{linenomath*}
from Bayes' theorem with prior density $\phi_\alpha$. 

The sampling probability $q(n_1,n_2)$ in \eqref{eq:qn1} and the resulting posterior density $\phi_\alpha^{(n_1,n_2)}$
play major roles in the two approaches we take to modeling latent mutations.  Specifically, transition probabilities in the conditional ancestral selection graph depend on ratios of sampling probabilities \citep{Slade2000b} and the allele frequency in the ancestral process of \citet{BartonEtAl2004} has initial density $\phi_\alpha^{(n_1,n_2)}(x)$ when conditioned on the sample.

We describe the occurrence of latent mutations in the ancestry of allele $A_1$ conditional on the sample count $n_1$.  We say that $A_1$ is \textit{rare} when the sample size $n$ is much larger than $n_1$.  We enforce this rarity of $A_1$ by letting $n_2 \sim n$ tend to infinity with $n_1$ fixed, or finite.  We present some results for cases in which $A_1$ is not rare in this sense, that is when neither $n_1$ nor $n_2$ is large. In this case we also describe the conditional ancestry of $A_2$, but overall our focus is on large samples and rare $A_1$.  This is the same, sample-based concept of rarity that was used in \citet{WakeleyEtAl2023} and previously considered by \citet{JoyceAndTavare1995} and \citet{Joyce1995}.  It may be distinguished from rarity in the population, though of course finding $A_1$ rare in a large sample is most likely when the population frequency $x$ is small. 

By \textit{strong selection} we mean large $\lvert\alpha\rvert$.  We model rarity and strong selection together under the assumption that $\alpha = \widetilde{\alpha} n_2$ for some constant $\widetilde{\alpha}\in\mathbb{R}$.  We study latent mutations and the ancestral processes which generate them under three scenarios: (i) $\lvert\alpha\rvert$ large with $n_2$ fixed, (ii) $n_2$ large with $\alpha$ fixed, and (iii) both $\lvert\alpha\rvert$ and $n_2$ large with $\widetilde{\alpha} = \alpha/n_2$ fixed.  In making approximations for large $n_2$ and/or large $\lvert\alpha\rvert$, we make extensive use of asymptotic results for ratios of gamma functions and for the confluent hypergeometric function which are presented in  \ref{sec:asymptotics}.

The parameters $\theta_1$ and $\theta_2$ are fixed constants throughout, with $\theta_1,\theta_2>0$.  For single nucleotide sites, these population-scaled mutation rates have been estimated for many species, using average pairwise sequence differences and assuming constant population size, and are typically about $0.01$ with a range of about $0.0001$ to $0.1$ \citep{LefflerEtAl2012}.  Values for humans are smaller but they vary almost as widely among sites in the genome, with a mean of about $0.0008$ and a range of about $0.0001$ to $0.02$ \citep{SeplyarskiyEtAl2021,SeplyarskiyEtAl2023,WakeleyEtAl2023}.  In contrast, there is no reason to suppose that the selection parameter $\lvert\alpha\rvert$ is small \citep{EyreWalkerAndKeightley2007,ChenEtAl2020,AgarwalEtAl2023}.  Note that our introduction of a constant $\widetilde{\alpha} = \alpha/n_2$ is simply a device to specify the relative importance of rarity as opposed strong selection, not a hypothesis about biology. 
  
The case of a rare neutral allele was considered in \citet{WakeleyEtAl2023} where it was shown that the number of latent mutations in the ancestry of the $n_1$ copies of allele $A_1$ follows the same distribution as the number of alleles in the Ewens sampling formula \citep{Ewens1972} with sample size $n_1$ and mutation parameter $\theta_1$.  Let $K_1 $ be the random number of these latent mutations for allele $A_1$ in the ancestry of the sample.  Further, let $\xi_j$ be a Bernoulli random variable with probability of success 
\begin{linenomath*}
\begin{equation}
\P(\xi_j=1) = \frac{\theta_1}{\theta_1+j-1} \quad , \quad j = 1,2,\ldots \label{eq:pcoalj}
\end{equation}
\end{linenomath*}
Under neutrality for large sample size and conditional on ${\mathcal N}_1=n_1$,
\begin{linenomath*}
\begin{equation}
K_1  \;\eqd\; \xi_{n_1} + \xi_{n_1-1} + \cdots + \xi_{2} + \xi_{1} \label{eq:K1sum}
\end{equation}
\end{linenomath*}
which gives the stated Ewens sampling result \citep{ArratiaBarbourAndTavare1992}. In \eqref{eq:K1sum} and below, $\eqd$ denotes equal in distribution. Note that, because coalescence is among exchangeable lineages, the full Ewens sampling formula should apply if we were to keep track of the sizes of latent mutations; see \citet{Crane2016} and \citet{Tavare2021} for recent reviews. 

Here we apply the model of coalescence in a random background described by \citet{BartonEtAl2004} to prove these results \eqref{eq:pcoalj} and \eqref{eq:K1sum} for rare alleles in large samples and especially to extend the analysis of latent mutations to scenarios which include selection.  We investigate both the number of latent mutations and their timing in the ancestry of the sample, and we allow that selection may be strong.  We also show how the same scenarios can be treated using the conditional ancestral selection graph \citep{Slade2000a}, giving the same limiting results for all three scenarios.  

Briefly, we find that positive selection does not in general lead to \eqref{eq:pcoalj} and \eqref{eq:K1sum}, that very strong positive selection (relative to the sample size) leads to neutral gene genealogies with a single ancient latent mutation for the favored allele.  This is described in Section~\ref{sec:bes} for scenario (i) and for the case $\widetilde{\alpha}\in(1,\infty)$ in scenario (iii).  On other hand, when selection is not too strong relative to the sample size, then extreme rarity of $A_1$ in the sample can effectively override strong positive selection and retrieve \eqref{eq:pcoalj} and \eqref{eq:K1sum}.  This is described in Section~\ref{sec:bes} for scenario (ii) and for the case $\widetilde{\alpha}\in(-\infty,1)$ in scenario (iii).  Figures~\ref{fig:scenario1}, \ref{fig:scenario2} and \ref{fig:scenario3b} illustrate our results in the three scenarios.

We note that \citet{FaveroAndJenkins2024arxiv} have recently performed detailed analysis of a $d$-allele diffusion model, where the selective advantage of one allele grows to infinity and the other parameters remain fixed. Their findings confirm and extend what we establish  for scenario (i) in the two-allele model  in Sections \ref{sec:bessub1} and \ref{sec:condasgsub1}.  In addition, \citet{FaveroAndJenkins2024arxiv} prove the duality of the strong-selection limit of the diffusion and the corresponding ancestral selection graph. 

\section{Sample frequencies and posterior population frequencies} \label{sec:freqs}  

In this section, we present asymptotic results for the sampling probability $q(n_1,n_2)$ in \eqref{eq:qn1} and the posterior density $\phi_\alpha^{(n_1,n_2)}(x)$ in \eqref{eq:condphi1} in our three regimes of interest: (i) $\lvert\alpha\rvert$ large with $n_2$ fixed, (ii) $n_2$ large with $\alpha$ fixed, and (iii) both $\lvert\alpha\rvert$ and $n_2$ large with $\widetilde{\alpha} = \alpha/n_2$ fixed.

\subsection{Asymptotics for sampling probabilities} \label{sec:samplefreqs}

In the case of strong selection and moderate sample size, that is $\lvert\alpha\rvert$ large with $n_2$ fixed, applying \eqref{eq:1F1s1} and \eqref{eq:1F1s2} to \eqref{eq:qn1} gives    
\begin{linenomath*}
\begin{subnumcases}{q(n_1,n_2) = }
\binom{n}{n_1} \frac{\Gamma(\theta_1+n_1)}{\Gamma(\theta_1)} {\lvert\alpha\rvert}^{-n_1} \left( 1 + O{\left({\lvert\alpha\rvert}^{-1}\right)} \right) & if $\,\alpha<0,$ \label{eq:qn1n2largealphaneg} \\[5pt]
\binom{n}{n_1} \frac{\Gamma(\theta_2+n-n_1)}{\Gamma(\theta_2)} {\alpha}^{n_1-n} \left( 1 + O{\left(\alpha^{-1}\right)} \right) & if $\,\alpha>0.$ \label{eq:qn1n2largealphapos}
\end{subnumcases} 
\end{linenomath*}
Here we focus on the leading-order terms but note that the next-order terms are straightforward to obtain using \eqref{eq:1F1s1} and \eqref{eq:1F1s2} and additional higher-order terms could be computed using (4.1.2) and (4.1.6) in \citet{Slater1960}.  In \eqref{eq:qn1n2largealphaneg}, each additional copy of $A_1$ decreases the sampling probability by a factor of $1/\lvert\alpha\rvert$ so the most likely sample is one which contains no copies of $A_1$.  In \eqref{eq:qn1n2largealphapos}, each additional copy of $A_1$ increases the sampling probability by a factor of $\alpha$ so the most likely sample is monomorphic for $A_1$.  However, these results are perfectly symmetric for the two alleles.  Switching allelic labels and swapping $\lvert\alpha\rvert$ for $\alpha$ changes \eqref{eq:qn1n2largealphaneg} into \eqref{eq:qn1n2largealphapos}.  That is, allele $A_2$ experiences the same effects of positive/negative selection in \eqref{eq:qn1n2largealphaneg}/\eqref{eq:qn1n2largealphapos} as the focal allele $A_1$ does in \eqref{eq:qn1n2largealphapos}/\eqref{eq:qn1n2largealphaneg}.  

In the case of large sample size and moderate selection, that is $n_2$ large with $\alpha$ fixed, applying \eqref{eq:1F1approxw} to \eqref{eq:qn1} gives  
\begin{linenomath*}
\begin{equation} 
q(n_1,n_2) = C \frac{\Gamma(\theta_1 + n_1)}{n_1 !} n_2^{-\theta_1} \left( 1 + O{\left(n_2^{-1}\right)} \right) . \label{eq:qn1n2largen2}
\end{equation} 
\end{linenomath*}
This has the same form as the neutral result, equation (22) in \citet{WakeleyEtAl2023}, only with the additional factor ${}_1F_1(\theta_1;\theta_1+\theta_2;\alpha)$ in the denominator of the constant $C$.  With respect to the count of the focal allele $A_1$, the distribution is similar to a (degenerate) negative-binomial distribution with parameters $p=1/n_2$ and $r=\theta_1$, like the corresponding result in Theorem 2 of \citet{Watterson1974a} for neutral alleles which propagate by a linear birth-death process.  The effect of selection is only to uniformly raise or lower the chances of seeing $n_1$ copies of $A_1$ in a very large sample.  The additional factor ${}_1F_1(\theta_1;\theta_1+\theta_2;\alpha)$ in the denominator of $C$ is a decreasing function of $\alpha$, which is equal to $1$ when $\alpha=0$ and approaches $0$ quickly from there as $\alpha$ increases.  Greater selection against (respectively, for) $A_1$ increases (respectively, decreases) the chance of it being rare but does not affect the shape of the distribution of $n_1$, at least to leading order in $1/n_2$.  

In the case of large sample size and strong selection, that is both $\lvert\alpha\rvert$ and $n_2$ large with $\widetilde{\alpha} = \alpha/n_2$ fixed, applying \eqref{eq:1F1s1}, \eqref{eq:1F1s2}, \eqref{eq:1F1ws1}, \eqref{eq:1F1ws3} to \eqref{eq:qn1} gives 
\begin{linenomath*}
\begin{subnumcases}{q(n_1,n_2) = }
\frac{\Gamma(\theta_1+n_1)}{n_1!\Gamma(\theta_1)} \left(\frac{1}{1+\lvert\widetilde{\alpha}\rvert}\right)^{n_1} \left(\frac{\lvert\widetilde{\alpha}\rvert}{1+\lvert\widetilde{\alpha}\rvert}\right)^{\theta_1} \left( 1 + O{\left(n_2^{-1}\right)} \right) & if $\,\widetilde{\alpha}<0$ \label{eq:qn1n2alphatilde1} \\[5pt]
B_1 \, \frac{\Gamma(\theta_1+n_1)}{n_1!} \left(\frac{1}{1-\widetilde{\alpha}}\right)^{n_1} \left( 1 + O{\left(n_2^{-1}\right)} \right) & if $\,0<\widetilde{\alpha}<1$ \label{eq:qn1n2alphatilde2} \\[5pt]
B_2 \, \frac{1}{n_1!} \left(\frac{\widetilde{\alpha}-1}{\widetilde{\alpha}}n_2\right)^{n_1} \left( 1 + O{\left(n_2^{-1}\right)} \right) & if $\,\widetilde{\alpha}>1$ \label{eq:qn1n2alphatilde3}
\end{subnumcases} 
\end{linenomath*}
with constants $B_1$ and $B_2$ which are unremarkable except in their dependence on $n_2$:
\begin{linenomath*}
\begin{align}
B_1 &\propto n_2^{\theta_2-\theta_1} e^{-\widetilde{\alpha}n_2} \notag \\[5pt]
B_2 &\propto n_2^{\theta_2-\frac{1}{2}} {\left(\widetilde{\alpha}e\right)}^{-n_2} \notag
\end{align}
\end{linenomath*}
such that $q(n_1,n_2)$ becomes tiny as $n_2$ grows.  In \eqref{eq:qn1n2alphatilde2} and \eqref{eq:qn1n2alphatilde3}, allele $A_1$ is favored by selection so it will be unlikely for its sample count to be very small. 

To see how \eqref{eq:qn1n2alphatilde2} and \eqref{eq:qn1n2alphatilde3} compare to \eqref{eq:qn1n2alphatilde1}, consider how these three sampling probabilities change as $n_1$ increases: 
\begin{linenomath*}
\begin{subnumcases}{\frac{q(n_1+1,n_2)}{q(n_1,n_2)} \approx }
\frac{\theta_1+n_1}{(1+\lvert\widetilde{\alpha}\rvert)(n_1+1)} & if $\,\widetilde{\alpha}<0$ \label{eq:qn1n2alphatilde1ratio} \\[5pt]
\frac{\theta_1+n_1}{(1-\widetilde{\alpha})(n_1+1)} & if $\,0<\widetilde{\alpha}<1$ \label{eq:qn1n2alphatilde2ratio} \\[5pt]
\frac{(\widetilde{\alpha}-1)n_2}{\widetilde{\alpha}(n_1+1)} & if $\,\widetilde{\alpha}>1$ \label{eq:qn1n2alphatilde3ratio} 
\end{subnumcases} 
\end{linenomath*}
where the approximation is for large $n_2$, i.e.\ omitting the $O{\left(n_2^{-1}\right)}$ parts of \eqref{eq:qn1n2alphatilde1}, \eqref{eq:qn1n2alphatilde2} and \eqref{eq:qn1n2alphatilde3}.  The first two differ from the corresponding neutral result $(\theta_1+n_1)/(n_1+1)$ by the constant factors $1/(1+\lvert\widetilde{\alpha}\rvert)<1$ in \eqref{eq:qn1n2alphatilde1ratio} and $1/(1-\widetilde{\alpha})>1$ in \eqref{eq:qn1n2alphatilde2ratio}.  Note that \eqref{eq:qn1n2largen2} gives the neutral result, as do \eqref{eq:qn1n2alphatilde1ratio} and \eqref{eq:qn1n2alphatilde2ratio} as $\widetilde{\alpha}\to0$.  Relative to this, negative selection in \eqref{eq:qn1n2alphatilde1ratio} makes additional copies of $A_1$ less probable whereas positive selection in \eqref{eq:qn1n2alphatilde2ratio} makes them more probable.  But \eqref{eq:qn1n2alphatilde1ratio} and \eqref{eq:qn1n2alphatilde2ratio} differ from the neutral result only by these constant factors.  Equation \eqref{eq:qn1n2alphatilde3ratio} is quite different.  With $\widetilde{\alpha}>1$, each additional copy of $A_1$ increases the sampling probability by a large factor, proportional to $n_2$, making this case similar to the case of strong positive selection in \eqref{eq:qn1n2largealphapos}.  This is as expected.  What is surprising is \eqref{eq:qn1n2alphatilde2ratio}, namely that strong selection ($\alpha\to\infty$) in favor of $A_1$ can be made to resemble neutrality simply by increasing the sample size relative to $n_1$.  

\subsubsection{Comparison to discrete Moran and Wright-Fisher models} \label{sec:relate}

We emphasize that our analyses in this work are of the Wright-Fisher diffusion model, given here as the SDE \eqref{eq:sde} with stationary density \eqref{eq:phix}.  It is of interest to know how well our results hold for discrete, exact models such as the Moran model and the Wright-Fisher model, especially as $\lvert\alpha\rvert\to\infty$ or $n\to\infty$ for finite $n_1$, in which cases we might expect the diffusion to be a relatively poor description of the dynamics.  In this section, we focus on \eqref{eq:qn1n2alphatilde1} and show that it can be obtained in a different way from a discrete-time Moran model, without first passing to the diffusion limit, but that this cannot be done in general starting from the discrete-time Wright-Fisher model.  

To leading order in $1/n_2$, \eqref{eq:qn1n2alphatilde1} is identical to the probability mass function of a negative binomial distribution with parameters $p=\lvert\widetilde{\alpha}\rvert/(1+\lvert\widetilde{\alpha}\rvert)$ and $r=\theta_1$.  \citet{CharlesworthAndHill2019} found this same result starting from the strong-selection approximation which \citet{Nei1968} had obtained for the diffusion model of \citet{Wright1937}.  Here selection \textit{against} $A_1$ is so strong that it never reaches appreciable frequency in the population.  In the limit, or ignoring the $O{\left(n_2^{-1}\right)}$ part in \eqref{eq:qn1n2alphatilde1}, this distribution sums to one over all $n_1\in\{0,1,2,\ldots\}$.  The corresponding sum for the degenerate distribution in \eqref{eq:qn1n2largen2} diverges, because under neutrality there is a non-trivial chance that $A_1$ reaches appreciable frequency in the population.   

Consider a discrete-time haploid Moran model with population size $N$, in which allele $A_2$ is favored by selection.  Specifically, $A_1$ and $A_2$ have equal chances of being chosen to reproduce but different chances of being chosen to die: each $A_1$ has an increased chance $1+s$ compared to each $A_2$.  Upon reproduction, the offspring of an $A_i$, $i\in\{1,2\}$, has type $A_i$ with probability $1-u_{3-i}$ and the other type $A_{3-i}$ with probability $u_{3-i}$. If there are currently $\ell$ copies of $A_1$ and $N-\ell$ copies of $A_2$, then in the next time step there will $\ell+1$ copies of $A_1$ with probability
\begin{linenomath*}
\begin{equation} 
\frac{N-\ell}{N-\ell+\ell(1+s)} \frac{\ell}{N} (1-u_2) + \frac{N-\ell}{N-\ell+\ell(1+s)} \frac{N-\ell}{N} u_1 \label{eq:pjplusone}
\end{equation}
\end{linenomath*}
and $\ell-1$ copies of $A_1$ with probability
\begin{linenomath*}
\begin{equation} 
\frac{\ell(1+s)}{N-\ell+\ell(1+s)} \frac{N-\ell}{N} (1-u_1) + \frac{\ell(1+s)}{N-\ell+\ell(1+s)} \frac{\ell}{N} u_2 . \label{eq:pjminusone}
\end{equation}
\end{linenomath*}
The fraction of $A_1$ converges to the Wright-Fisher diffusion process \eqref{eq:sde} as $N\to\infty$ if time is measured in units of $N(N-1)/2$ discrete steps, i.e. $dt=2/N(N-1)$, with $u_1=\theta_1/N$, $u_2=\theta_2/N$ and $s=-\alpha/N$.

As another way of obtaining \eqref{eq:qn1n2alphatilde1}, we assume that $s \gg u_1,u_2$.  In particular, let $Nu_1 \to \theta_1$ and $Nu_2 \to \theta_2$ as $N\to\infty$ just as in the diffusion model, but let $s$ be a constant.  Then we may appeal to the analogous model and limit process (iii) of \citet{KarlinAndMcGregor1964} which had no selection but instead assumed that $u_2 \gg u_1$.  Similarly here we expect that allele $A_1$ will be held in negligible relative frequency in the population and instead be present in a finite number of copies as $N\to\infty$, only here due to strong selection rather than strong mutation. 

In view of this scaling of the mutation rates by $N$ and for comparison with the Wright-Fisher model below, we rescale time so that it is measured in unit of generations, or $N$ time steps.  Then with $dt=1/N$, we can rewrite \eqref{eq:pjplusone} and \eqref{eq:pjminusone} as
\begin{linenomath*}
\begin{equation*} 
\left( \ell \frac{N-\ell}{N-\ell+\ell(1+s)} (1-u_2) + \frac{N-\ell}{N-\ell+\ell(1+s)} (N-\ell) u_1 \right) dt , 
\end{equation*}
\end{linenomath*}
and 
\begin{linenomath*}
\begin{equation*} 
\left( \ell(1+s) \frac{N-\ell}{N-\ell+\ell(1+s)} (1-u_1) + \frac{\ell(1+s)}{N-\ell+\ell(1+s)} \ell u_2 \right) dt . 
\end{equation*}
\end{linenomath*}
Then in the limit $N\to\infty$, \eqref{eq:pjplusone} and \eqref{eq:pjminusone} describe to a continuous-time process in which 
\begin{linenomath*}
\begin{equation}
\ell \to 
  \begin{cases}
    \ell+1 & \textrm{at rate } \,\, {\left( \ell + \theta_1 \right)} \quad  \\[4pt]
    \ell-1 & \textrm{at rate } \,\, \ell (1+s) . \quad  
  \end{cases} \label{eq:bdell}
\end{equation}
\end{linenomath*}
In other words, the number of copies of $A_1$ in the population evolves according to a birth-death process with immigration where the birth rate is $\lambda = 1$, the death rate is $\mu = 1+s$ and the immigration rate is $\kappa = \theta_1$.  From (52) in \citet{Kendall1949}, the distribution of the number of copies of $A_1$ in the population at stationarity will be negative binomial with parameters $1-\lambda/\mu = s/(1+s)$ and $\kappa/\lambda = \theta_1$, or 
\begin{linenomath*}
\begin{equation} 
p(\ell) = \binom{\ell+\theta_1-1}{\ell} {\left(\frac{1}{1+s}\right)}^\ell {\left(\frac{s}{1+s}\right)}^{\theta_1}, \qquad \ell\in\Z_+. \label{eq:pell}
\end{equation}
\end{linenomath*}

In getting to \eqref{eq:qn1n2alphatilde1} above, which we note is for $\alpha<0$, we first applied the diffusion limit then let the selection parameter $\lvert\alpha\rvert$ be large, specifically proportional to the number $n_2$ of copies of $A_2$ in a sample of large size $n$ for a given fixed number $n_1$ of copies of $A_1$.  In the current haploid Moran model with selection against $A_1$, $\lvert\alpha\rvert=Ns$, and the scalar $\lvert\widetilde{\alpha}\rvert=\lvert\alpha\rvert/n_2=Ns/n_2$.  Define $a \coleq n_2/N$.  Then $\lvert\widetilde{\alpha}\rvert=s/a$ and we may think of $a$ as (close to) the proportion of the population sampled, because $n_2/N \sim n/N$. 

If the number of copies of $A_1$ in the population is $\ell$, then the probability there are $n_1$ copies in a sample of size $n$ taken without replacement from the total population of size $N$ is given by the hypergeometric distribution 
\begin{linenomath*}
\begin{equation} 
p(n_1 \vert \ell ; N) = \frac{\binom{\ell}{n_1}\binom{N-\ell}{n-n_1}}{\binom{N}{n}}, \qquad n_1=0,1,2,\ldots,\ell . \label{eq:pn1givenellN}
\end{equation}
\end{linenomath*}
Since $n-n_1=n_2=a N$ and taking $N \to \infty$, \eqref{eq:pn1givenellN} converges to the binomial distribution  
\begin{linenomath*}
\begin{equation} 
p(n_1 \vert \ell) = \binom{\ell}{n_1} a^{n_1} (1-a)^{\ell-n_1}, \qquad n_1=0,1,2,\ldots,\ell. \label{eq:pn1givenell}
\end{equation}
\end{linenomath*}
This gives another route to \eqref{eq:qn1n2alphatilde1}, namely using \eqref{eq:pell} and \eqref{eq:pn1givenell}, and setting $y=\ell-n_1$,
\begin{linenomath*}
\begin{align} 
p(n_1) &= \sum_{\ell=n_1}^\infty p(n_1 \vert \ell) p(\ell) \notag \\[5pt]
 &= \sum_{\ell=n_1}^\infty \binom{\ell}{n_1} a^{n_1} (1-a)^{\ell-n_1} \binom{\ell+\theta_1-1}{\ell} {\left(\frac{1}{1+s}\right)}^\ell {\left(\frac{s}{1+s}\right)}^{\theta_1} \notag \\[5pt]
&= \frac{1}{n_1!} {\left(\frac{a}{1+s}\right)}^{n_1} {\left(\frac{s}{a+s}\right)}^{\theta_1} \sum_{y=0}^\infty \frac{\Gamma(y+n_1+\theta_1)}{\Gamma(y+\theta_1)} \binom{y+\theta_1-1}{y} {\left(\frac{1-a}{1+s}\right)}^{y} {\left(\frac{a+s}{1+s}\right)}^{\theta_1} \notag \\[5pt]
&= \frac{\Gamma(n_1+\theta_1)}{n_1!\Gamma(\theta_1)} {\left(\frac{a}{a+s}\right)}^{n_1} {\left(\frac{s}{a+s}\right)}^{\theta_1} . \label{eq:altqn1n2alphatilde1}
\end{align}
\end{linenomath*}
The end result \eqref{eq:altqn1n2alphatilde1} is equal to the leading order part of \eqref{eq:qn1n2alphatilde1} since $\lvert\widetilde{\alpha}\rvert=s/a$. 

We can contrast this with a discrete-time haploid Wright-Fisher model with population size $N$, in which time is already measured in generations.  Under the same assumptions that gave \eqref{eq:bdell}, namely $s$ constant and $u_1,u_2\propto1/N$ as $N\to\infty$, we can use equation (33) in \citet{Nagylaki1990} which specifies that, conditional on the number $\ell_g$ of copies of $A_1$ in generation $g$, the number $\ell_{g+1}$ has the Poisson distribution
\begin{linenomath*}
\begin{equation} 
\ell_{g+1} \vert \ell_{g} \sim \textrm{Poisson}{\left(\frac{\theta_1}{2} + \ell_g {\left(1-\frac{s}{2}\right)}\right)} \label{eq:nag33}
\end{equation}
\end{linenomath*}
with $\theta_1 \coleq 2Nu_1$.  Now $\ell$ evolves by a Poisson branching process with Poisson immigration rather than by the birth-death process with immigration in \eqref{eq:bdell}.  Although here too the number of copies of $A_1$ in the population will converge to a stationary distribution \citep{Heathcote1965}, it will not in general be a negative binomial distribution.  Thus \eqref{eq:qn1n2alphatilde1} is consistent with the per-generation dynamics of rare alleles in the Moran model but not in the Wright-Fisher model.  

\subsection{Asymptotics for population frequencies conditional on the sample} \label{sec:populationfreqs}

Next, we obtain asymptotics for the posterior probability density $\phi_\alpha^{(n_1,n_2)}$ in \eqref{eq:condphi2}.  Let $\mathcal{P}([0,1])$ be the space of probability measures on $[0,1]$ endowed with the weak convergence topology (i.e with test functions in the space $C_b([0,1])$ of bounded continuous functions on $[0,1]$).
\begin{lemma}\label{L:Convphi}
Let $n_1\in\mathbb{N}$ be fixed. The following  convergences in $\mathcal{P}([0,1])$ hold.
\begin{itemize}
\item[(i)] Suppose $n_2\in\mathbb{N}$ is fixed and $\alpha\to\infty$. Then $\phi^{(n_1,n_2)}_{\alpha}(x)\,dx \to \delta_1$.
\item[(ii)] Suppose $\alpha\in\R$ is fixed and $n_2\to\infty$. Then  $\phi^{(n_1,n_2)}_{\alpha}(x)\,dx \to \delta_0$.
\item[(iii)] Suppose $\alpha=\widetilde{\alpha}n_2+c$ where $\widetilde{\alpha},\,c\in\R$ are fixed and $n_2\to\infty$. Then 
\begin{linenomath*}
\begin{align}\label{Dirac_posterior}
\phi^{(n_1,n_2)}_{\alpha}(x)\,dx &\to
\begin{dcases}
\delta_0 & \textrm{when } \quad  \widetilde{\alpha}\in (-\infty,1] \\[6pt]
\delta_{1-1/\widetilde{\alpha}} & \textrm{when } \quad  \widetilde{\alpha}\in (1,\infty) 
\end{dcases}
\end{align}
\end{linenomath*}
\end{itemize}
where $\delta_x$ is the Dirac delta measure.
\end{lemma}

The proof of Lemma \ref{L:Convphi} is given in \ref{sec:lemmaproofs}.

\section{Conditional coalescence in a random background} \label{sec:bes}

In this section, we extend the approach of coalescence in a random background in \citet{BartonEtAl2004} to study the number and timing of latent mutations and other asymptotic properties of the conditional gene genealogy given the sample frequencies of $A_1$ and $A_2$. We also extend our results to  time-varying populations in Section \ref{S:varying}.  While the setting of \citet{BartonEtAl2004} covers the case of a neutral locus linked to the selected locus, here we focus on the selected locus.

Suppose we are given a sample from the selected locus at the present time $t=0$, and that we know the allelic types of the sample but we do not know how the sample was produced.  What is the genealogy of the sample?  This question was answered by \citet{BartonEtAl2004}, who modeled the ancestral process using the structured coalescent with allelic types as subpopulations.  The structured coalescent can be a model of subdivision with migration between local populations \citep{Takahata1988,Notohara1990,Herbots1997} or a model of selection with mutation between allelic types \citep{KaplanEtAl1988,DardenEtAl1989}.  For samples from a population at stationarity as in Section~\ref{sec:intro}, \citet{BartonEtAl2004} proved that this could be done rigorously starting with a Moran model with finite $N$ then passing to the diffusion limit.  \citet{BartonAndEtheridge2004} explored some properties of gene genealogies under this model, and \citet{EtheridgeEtAl2006} used the same idea to describe genetic ancestries following a selective sweep. 

Even if the sample frequencies are known, the allele frequencies in the population are unknown.  A key feature of this method is to model allele-frequency trajectories backward in time.  As pointed out by \citet{BartonEtAl2004}, the Moran model with finite $N$ is \textit{reversible}, meaning that at stationarity the time-reversed process is the same (in distribution) as the forward-time Moran process.  This is not a property of the Wright-Fisher model with finite $N$ but does hold for their shared diffusion limit \eqref{eq:sde} with stationary density \eqref{eq:phix}; see for instance \cite{millet1989integration} for why this holds.  Figure \ref{fig:scenario1} gives an illustration of a genealogy with mutations and allele frequencies varying backward in time.

Looking backward in time, let $p^{(N)}_t$ be the fraction of type 1 in the population and $n^{(N)}_i(t)$ be the number of ancestral lineages of type $i\in\{1,2\}$ at time $t$.  From \citet[Lemma 2.4]{BartonEtAl2004}, under the Moran model with stationary distribution, $(p^{(N)}_t,\,n^{(N)}_1(t),\,n^{(N)}_2(t))_{t\in\R_+}$ is a Markov process for each fixed $N$. Furthermore, \citet[Theorem 5.1]{BartonEtAl2004} describes the joint convergence of the processes  $(p^{(N)},\,n^{(N)}_1,\,n^{(N)}_2)$ as $N\to\infty$. 

\begin{lemma}[Lemma 2.4 and Theorem 5.1 of \citet{BartonEtAl2004}]\label{L:Diff_Moran}
Let $p^{(N)}_t$ and $n^{(N)}_i(t)$  be the fraction of type 1 in the population and  the number of ancestral lineages of type $i$ respectively, at time $t$ backward, under the stationary Moran model. Then $(p^{(N)},\,n^{(N)}_1,\,n^{(N)}_2)$ is a Markov process for each $N\in\mathbb{N}$. As $N\to\infty$, this process converges in distribution in the Skorohod space $D(\R_+,\,[0,1]\times\Z_+\times \Z_+)$ to a Markov process 
$(p_t,\,n_1(t),\,n_2(t))_{t\in\R_+}$ described as follows: 
\begin{itemize}
\item[(i)] $\,(p_t)_{t\in\R_+}$ is a solution to equation \eqref{eq:sde} with stationary initial density $\phi_\alpha$. In particular, it does not depend on $(n_1(t),\,n_2(t))_{t\in\R_+}$.
\item[(ii)] Suppose the current state is $(p,m_1,m_2)$. Then $(n_1(t),\,n_2(t))_{t\in\R_+}$ evolves as
\begin{linenomath*}
\begin{align}\label{E:genealogy_Moran}
(m_1,m_2) &\to 
\begin{cases}
(m_1-1,m_2) & \textrm{at rate} \quad \frac{1}{p}\binom{m_1}{2}  \quad \qquad \text{ coalescence of type 1} \\[6pt]
(m_1-1,m_2+1) & \textrm{at rate} \quad \frac{1-p}{p}\,m_1\,\frac{\theta_1}{2} \qquad \text{ mutation of 1 to 2}\\[6pt]
(m_1,m_2-1) & \textrm{at rate} \quad \frac{1}{1-p}\binom{m_2}{2} \qquad \text{ coalescence of type 2} \\[6pt]
(m_1+1,m_2-1) & \textrm{at rate} \quad \frac{p}{1-p}\,m_2\,\frac{\theta_2}{2} \qquad \text{ mutation of 2 to 1}\\[6pt]
\end{cases}
\end{align}
\end{linenomath*}
\end{itemize}
\end{lemma}

We compare the notation here with that of \citet{BartonEtAl2004} and \citet{etheridge2011some}. Type 1 here is their type P, type 2 here is their type Q. So,  $\frac{\theta_2}{2}$ (resp. $\frac{\theta_1}{2}$) here is $\mu_1$ (resp. $\mu_2$) in \citet{BartonEtAl2004}, and is  $v_1$ (resp. $v_2$) in \citet[eqn. (2.11)]{etheridge2011some}. The rescaled selection coefficient $\frac{\alpha}{2}$  here is $s$ in \citet{BartonEtAl2004}. 
The Moran model in \citet{BartonEtAl2004} has population size $2N$ and each of the $(2N)(2N-1)/2$ unordered pairs is picked to interact (one dies and immediately the other reproduces) at rate $1/2$.
Therefore, the diffusion limit in \citet[Lemma 3.1]{BartonEtAl2004} is slower by a factor of $1/2$ than the limit we consider here in this paper. 

The proof of \citet[Theorem 5.1]{BartonEtAl2004} leads to more information for the limiting process. Let  $n^{(N),obs}_i(t)$ be the number of type $i$ lineages at backward time $t$ which are ancestral to the $n^{(N)}_i(0)$ observed in the sample. 
Thus $n^{(N),obs}_i(t)$ is non-increasing in $t$, but $n^{(N)}_i(t)$ can increase as $t$ increases due to mutations from type $3-i$ to type $i$. Clearly, $n^{(N),obs}_i(0)=n^{(N)}_i(0)$.
Note that a mutation from type $3-i$ to type $i$ backward in time corresponds to a mutation from type $i$ to type $3-i$ forward in time.
To keep track of the number of mutation events versus the number of coalescent events, we let $L^{(N)}_i(t)$ be the total number of latent mutations for type $i$ during $[0,t]$ backward in time. 
We have the following generalization of Lemma \ref{L:Diff_Moran}.
\begin{lemma}[Joint convergence]\label{L:Diff_Moran_Gen}
Under the stationary Moran model, the backward process $$(p^{(N)};\,n^{(N)}_1,n^{(N),obs}_1, L^{(N)}_1;\,n^{(N)}_2,n^{(N),obs}_2, L^{(N)}_2)$$ is a Markov process for each fixed $N\in\mathbb{N}$. As $N\to\infty$, this process converges in distribution in the Skorohod space $D(\R_+,\,[0,1]\times\Z_+^6)$ to a continuous-time Markov process 
$$(p_t;\,n_1(t),n^{obs}_1(t), L_1(t);\,n_2(t),n^{obs}_2(t), L_2(t))_{t\in\R_+}$$  
such that 
\begin{itemize}
\item[(i)] $\,(p_t)_{t\in\R_+}$ is a solution to equation \eqref{eq:sde}  with stationary initial density $\phi_\alpha$. In particular, it does not depend on $(n_1(t),n^{obs}_1(t), L_1(t);\,n_2(t),n^{obs}_2(t), L_2(t))_{t\in\R_+}$ .
\item[(ii)] At state  $(p;\,m_1, a_1, \ell_1;\;m_2, a_2, \ell_2)$, the process $(n_1(t),n^{obs}_1(t), L_1(t);\,n_2(t),n^{obs}_2(t), L_2(t))_{t\in\R_+}$ evolves as $(m_1, a_1, \ell_1;\;m_2, a_2, \ell_2) \to$
\begin{linenomath*}
\begin{align*}
\begin{cases}
(m_1-1, a_1-1, \ell_1;\;m_2, a_2, \ell_2) & \textrm{at rate} \quad \frac{1}{p}\binom{a_1}{2}  \qquad \qquad \quad \text{ coalescence of two type 1}^{obs} \\[6pt]
(m_1-1, a_1, \ell_1;\;m_2, a_2, \ell_2) & \textrm{at rate} \quad \frac{1}{p}\left[\binom{m_1}{2}-\binom{a_1}{2} \right]  \quad \text{ other coalescence of  type 1} \\[6pt]
(m_1-1, a_1-1, \ell_1+1;\;m_2+1, a_2, \ell_2)  & \textrm{at rate} \quad \frac{1-p}{p}\,a_1\,\frac{\theta_1}{2} \qquad\qquad \text{ mutation of 1}^{obs}\\[6pt]
(m_1-1, a_1, \ell_1;\;m_2+1, a_2, \ell_2)  & \textrm{at rate} \quad \frac{1-p}{p}\,(m_1-a_1)\,\frac{\theta_1}{2} \quad \text{ mutation of other type 1}
\end{cases}
\end{align*}
\end{linenomath*}
and, similarly,
\begin{linenomath*}
\begin{align*}
\begin{cases}
(m_1, a_1, \ell_1;\;m_2-1, a_2-1, \ell_2) & \textrm{at rate} \quad \frac{1}{1-p}\binom{a_2}{2}  \qquad \qquad \quad \text{ coalescence of two type 2}^{obs} \\[6pt]
(m_1, a_1, \ell_1;\;m_2-1, a_2, \ell_2) & \textrm{at rate} \quad \frac{1}{1-p}\left[\binom{m_2}{2}-\binom{a_2}{2} \right]  \quad \text{ other coalescence of type 2} \\[6pt]
(m_1+1, a_1, \ell_1;\;m_2-1, a_2-1, \ell_2+1)  & \textrm{at rate} \quad \frac{p}{1-p}\,a_2\,\frac{\theta_2}{2} \qquad\qquad \text{ mutation of type 2}^{obs}\\[6pt]
(m_1+1, a_1, \ell_1;\;m_2-1, a_2, \ell_2)  & \textrm{at rate} \quad \frac{p}{1-p}\,(m_2-a_2)\,\frac{\theta_2}{2} \quad \text{ mutation of other type 2}
\end{cases}
\end{align*}
\end{linenomath*}
\end{itemize}
\end{lemma}

\medskip

Now suppose that in addition to knowing the sample counts $n_1$ and $n_2$, we also know that these are the outcome of uniform random sampling, as in \eqref{eq:qn1}.
Let $\P_{\bf n}$ be the conditional probability measure of the ancestral process in Lemma \ref{L:Diff_Moran_Gen} including both $p_t$ and the lineage dynamics, given that a uniformly picked sample has allelic counts ${\bf n}=(n_1,n_2)$. 
Under $\P_{\bf n}$, the limiting process in Lemma \ref{L:Diff_Moran_Gen} has initial frequency $p_0\sim \phi^{(n_1,n_2)}_{\alpha}(x)\,dx$ given by \eqref{eq:condphi2}, 
and  $(n_1(0),n^{obs}_1(0), L_1(0);\,n_2(0),n^{obs}_2(0), L_2(0))=(n_1,n_1,0;\,n_2,n_2,0)$. This follows from Bayes' theorem, because
by part (i) of  Lemmas \ref{L:Diff_Moran} and  \ref{L:Diff_Moran_Gen}, the initial frequency has prior density given by \eqref{eq:phix}.

Focus on type 1 for now.
We care about the sequence of events (coalescence and mutation) backward in time for type 1, and the timing of these events. At each of these events, the number of  type 1 lineages decreases by 1, either by coalescence or by mutation from type 1 to type 2. Hence $n^{obs}_1(t)$ is non-increasing, but $n_1(t)$ can increase over time (backward) due to mutations from type 2 to type 1 (see Figure \ref{fig:scenario1} for an illustration). Furthermore, the difference $n_1(t)- n^{obs}_1(t)$ is the number of type 1 at time $t$ that came from lineages that are of type 2 in the sample (at $t=0$).
We do not care about these $n_1(t)- n^{obs}_1(t)$ lineages, nor  the mutation events from type 2 to type 1. Analogous considerations hold for type 2.

From Lemma \ref{L:Diff_Moran_Gen}, we immediately obtain the following simplified description for the conditional ancestral process in the limit $N\to\infty$ for the two types. This description is the starting point of our  analysis for constant population size; later in Proposition \ref{prop:BothType_vary} we also obtain the analogous result for time-varying population size.
\begin{prop}[Conditional ancestral process]\label{prop:BothType}
The process $(p_t,\,n^{obs}_1(t),\,L_1(t),\,n^{obs}_2(t),\,L_2(t))_{t\in\R_+}$ under $\P_{\bf n}$ is a Markov process with state space $[0,1]\times \{0,1,\cdots,n_1\}^2 \times  \{0,1,\cdots,n_2\}^2$ described as follows: 
\begin{itemize}
\item[(i)] $\,(p_t)_{t\in\R_+}$ is a solution to \eqref{eq:sde} with initial density $\phi^{(n_1,n_2)}_{\alpha}$. In particular, it does not depend on the process $(n^{obs}_1,\,L_1,\,n^{obs}_2,\,L_2)$.
\item[(ii)] The process $(n^{obs}_1,\,L_1,\,n^{obs}_2,\,L_2)$ starts at $(n_1,0,n_2,0)$.
When the current state is $(p,a_1,\ell_1,a_2,\ell_2)$, this process evolves as
\begin{linenomath*}
\begin{align}\label{E:Evo_type1}
(a_1,\ell_1) &\to 
\begin{dcases}
(a_1-1,\ell_1) & \textrm{at rate} \quad \frac{1}{p}\binom{a_1}{2}  \quad \qquad \text{ coalescence of type 1}^{obs} \\[6pt]
(a_1-1,\ell_1+1) & \textrm{at rate} \quad \frac{1-p}{p}\,a_1\,\frac{\theta_1}{2} \qquad  \text{ mutation of 1}^{obs}\text{ to 2}\\[6pt]
\end{dcases}
\end{align}
\end{linenomath*}
and, independently,
\begin{linenomath*}
\begin{align}\label{E:Evo_type2}
(a_2,\ell_2) &\to 
\begin{dcases}
(a_2-1,\ell_2) & \textrm{at rate} \quad \frac{1}{1-p}\binom{a_2}{2}  \quad \qquad \text{ coalescence of type 2}^{obs} \\[6pt]
(a_2-1,\ell_2+1) & \textrm{at rate} \quad \frac{p}{1-p}\,a_2\,\frac{\theta_2}{2} \qquad  \text{ mutation of 2}^{obs}\text{ to 1}\\[6pt]
\end{dcases}
\end{align}
\end{linenomath*}
\end{itemize}
\end{prop}

\medskip

The total rate in \eqref{E:Evo_type1}, at which $a_1$ decreases by 1, is 
\begin{linenomath*}
\begin{equation}\label{taurate}
\frac{a_1}{2p}\big((1-p)\theta_1+a_1-1\big)\, \eqcol \,\lambda_{a_1}(p),
\end{equation}
\end{linenomath*}
and the one-step transition probabilities are
\begin{linenomath*}
\begin{align} \label{tauprob}
(a_1,\ell_1) &\to 
\begin{dcases}
(a_1-1,\ell_1) & \textrm{with probability} \quad \frac{a_1-1}{(1-p)\theta_1+a_1-1}  \\[6pt]
(a_1-1,\ell_1+1) & \textrm{with probability} \quad \frac{(1-p)\theta_1}{(1-p)\theta_1+a_1-1} \eqcol h_{a_1}(p)  \\[6pt]
\end{dcases}
\end{align}
\end{linenomath*}

As $t$ increases from 0 to $\infty$, the process $(n^{obs}_1(t))_{t\in\R_+}$ decreases from $n_1$ to 0, and the process $(L_1(t))_{t\in\R_+}$ increases from 0 to a random number
$K_1\coleq\lim_{t\to\infty}L_1(t)\in\mathbb{N}$ which  is the total number of latent mutations for type 1. Similarly, the total number of latent mutations for type 2 is defined by $K_2\coleq\lim_{t\to\infty}L_2(t)\in\mathbb{N}$.

We now give a more explicit description of $(K_1,K_2)$ using the frequency process $p$ and independent Bernoulli random variables.  Note these are conditional on the sample counts $({\mathcal N}_1=n_1,{\mathcal N}_2=n_2)$ as in \eqref{eq:K1sum}.

Let $\tau_1<\tau_2<\cdots< \tau_{n_1}$ be the jump times of the process $n^{obs}_1$. At time $\tau_1$, the process $n^{obs}_1$ decreases from $n_1$ to $n_1-1$, etc., until finally at $\tau_{n_1}$, $n^{obs}_1$ decreases from $1$ to $0$. It can be checked that $\tau_{n_1}<\infty$ almost surely under $\P_{\bf n}$ using the ergodicity of the process $p$ and \eqref{taurate}. Thus $(n^{obs}_1(t))_{t\in\R_+}$ will indeed decrease to 0 eventually under $\P_{\bf n}$.  
The allele frequencies at these random times are  $p_{\tau_1},\,p_{\tau_2},\ldots,\,p_{\tau_{n_1}}$.
By Proposition \ref{prop:BothType}, under $\P_{\bf n}$, we have 
\begin{linenomath*}
\begin{align}
K_1 \;\eqd & \; \xi_{n_1}(p_{\tau_1})+\xi_{n_1-1}(p_{\tau_2})+\xi_{n_1-2}(p_{\tau_3}) +\cdots+ \xi_{2}(p_{\tau_{n_1-1}}) +1, \label{Latent_type1a}
\end{align}
\end{linenomath*}
where $\{\xi_{k}(\cdot)\}_{k=2}^{n_1}$ is a family of independent random processes such that, for a constant $p\in[0,1]$,
\begin{linenomath*}
\begin{align} 
\xi_{k}(p) &=
\begin{dcases}
0 & \textrm{with probability} \quad \frac{k-1}{(1-p)\theta_1+k-1 }  \\[6pt]
1 & \textrm{with probability} \quad \frac{(1-p)\theta_1}{(1-p)\theta_1+k-1 }  
\end{dcases}\label{xi_K1}
\end{align}
\end{linenomath*}
which is a generalization of $\xi_k$ in \eqref{eq:pcoalj} and \eqref{eq:K1sum}, where $\xi_k \equiv \xi_{k}(0)$.

Similarly, if we let $s_1<s_2<\cdots< s_{n_2}$ be the jump times of the process $n^{obs}_2$, then 
\begin{linenomath*}
\begin{align}
K_2 \;\eqd & \; \zeta_{n_2}(p_{s_1})+\zeta_{n_2-1}(p_{s_2})+\zeta_{n_2-2}(p_{s_3}) +\cdots+ \zeta_{2}(p_{s_{n_2-1}}) +1 \label{Latent_type2a},
\end{align}
\end{linenomath*}
where $\{\zeta_{k}(\cdot)\}_{k=2}^{n_2}$ is a family of independent random processes such that, for a constant $p\in[0,1]$,
\begin{linenomath*}
\begin{align*} 
\zeta_{k}(p) &=
\begin{dcases}
0 & \textrm{with probability} \quad \frac{k-1}{p\theta_2+k-1 }  \\[6pt]
1 & \textrm{with probability} \quad \frac{p\theta_2}{p\theta_2+k-1 } 
\end{dcases}
\end{align*}
\end{linenomath*}
Analogous to $\xi_k$ in reference to $K_1$, in what follows we will use the notation
\begin{linenomath*}
\begin{equation}
\zeta_k \equiv \zeta_{k}(1) \label{eq:zetakdef}
\end{equation}
\end{linenomath*} 
in reference to $K_2$.

Having described the ancestral process in Proposition \ref{prop:BothType} and the latent mutations in \eqref{Latent_type1a} and \eqref{Latent_type2a} under the conditional probability $\P_{\bf n}$, we study their asymptotic properties under  3 scenarios in the next 3 subsections.
These 3 scenarios are a consequence of the asymptotic behaviors of the initial frequency $p_0$ described in Lemma \ref{L:Asymp_Initial}.
\begin{lemma}[Asymptotic initial frequency]\label{L:Asymp_Initial}
Let $n_1\in\mathbb{N}$ be fixed. The initial frequency $p_0$ converges in probability under  $\P_{\bf n}$ to a deterministic constant as follows.
\begin{itemize}
\item[(i)] Suppose $n_2\in \mathbb{N}$ is fixed and $\alpha\to\infty$. Then  $p_0\to 1$.
\item[(ii)] Suppose $\alpha\in \R$ is fixed and $n_2\to\infty$. Then  $p_0\to 0$.
\item[(iii)] Suppose $\alpha=\widetilde{\alpha}n_2+c$ where $\widetilde{\alpha},\,c\in\R$ are fixed and $n_2\to\infty$. Then 
\begin{linenomath*}
\begin{align}\label{Dirac_initial}
p_0 &\to
\begin{dcases}
0 & \textrm{when } \quad  \widetilde{\alpha}\in (-\infty,1] \\[6pt]
1-1/\widetilde{\alpha} & \textrm{when } \quad  \widetilde{\alpha}\in (1,\infty) 
\end{dcases}
\end{align}
\end{linenomath*}
Furthermore, when $\widetilde{\alpha}\in (-\infty,1)$,  
it holds that $n_2p_0$ converges in distribution to the Gamma random variable ${\rm Gam}(n_1+\theta_1,1-\widetilde{\alpha})$ with probability density function $\frac{ (1-\widetilde{\alpha})^{(n_1+\theta_1)}}{\Gamma(n_1+\theta_1)}\,y^{n_1+\theta_1 - 1} \,e^{(\widetilde{\alpha}-1) y}$.
\end{itemize}
\end{lemma}

The proof of Lemma \ref{L:Asymp_Initial} is given in \ref{sec:lemmaproofs}.

\subsection{Scenario (i): strong selection, arbitrary sample size} \label{sec:bessub1}

The first  scenario is when $\lvert\alpha\rvert$ large with $n_2$ fixed. We consider the case $\alpha\to+\infty$ only, since the other case $\alpha\to-\infty$ follows by switching the roles of type 1 and type 2. 

The conditional genealogy of the $n_1+n_2$ sampled individuals, under  $\P_{\bf n}$,  has three parts with different timescales. 
First, the $n_2$ type-2 lineages quickly evolve (coalesce and mutate) as in the  Ewens sampling formula, producing $K_2$ type 1 lineages at a short time $s_{n_2}$. Thus,
$K_2\eqd \sum_{k=1}^{n_2}\zeta_k$ where $\{\zeta_k\}$ are independent Bernoulli variables taking values in $\{0,1\}$ and having means $\frac{\theta_2}{\theta_2+k-1}$.
Next, the resulting  $n_1+K_2$ type 1 lineages will coalesce according to the Kingman coalescent without mutation until only one lineage remains. Hence it takes $O(1)$ amount of time for the number of lineages of type 1 to decrease to 1, as $\alpha\to\infty$.  Finally, it takes a long time,  $\tau_{n_1} \approx \frac{2\,\alpha}{\theta_1\theta_2}$, for the single lineage to mutate.
In particular, $K_1\approx 1$.

This description is justified by Theorems \ref{T:scenario1} and \ref{T:Age1}. See Figure \ref{fig:scenario1} for an illustration.

\FloatBarrier
\begin{figure}[h!]
    \centering
    \includegraphics[scale=0.6]{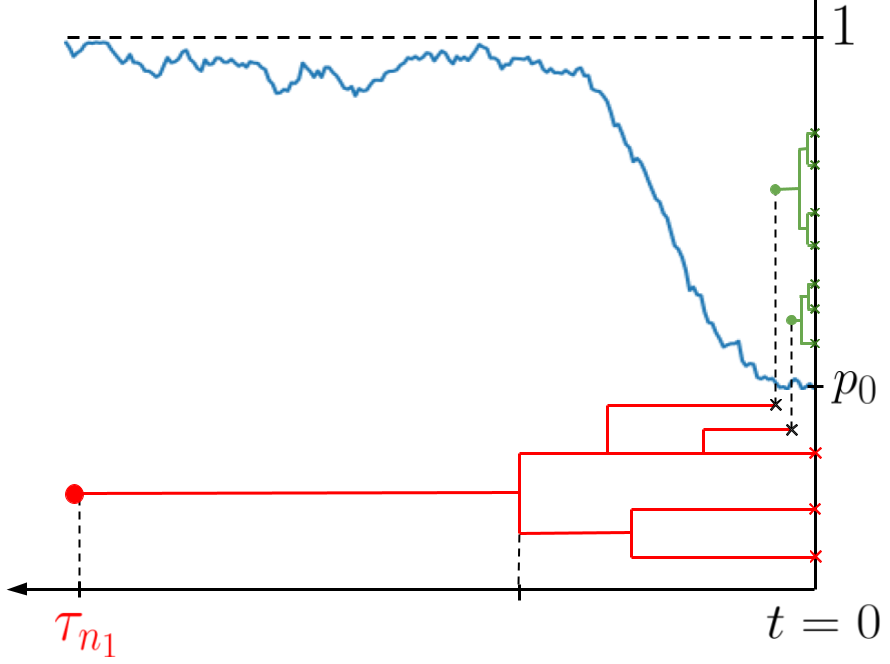}
    \caption{Conditional genealogy of a sample with  $(n_1,n_2)=(3,7)$ at the present time $t=0$. The fluctuating blue curve shows the process $(p_t)_{t\in\R_+}$ of the population frequency of type 1 backward in time. In this example, $p_t$ approaches $1$ from $p_0$ and the $7$ type-2 lineages coalesce and mutate, producing an additional $K_2=2$ type-1 lineages. The $5=3+2$ type-1 lineages then coalesce without mutating, reaching their common ancestral lineage at time $\tau_{n_1-1}$ and finally mutating at time $\tau_{n_1}$.  Under scenario (i), that is when $\alpha$ is large: $p_0$ will already be close to $1$, coalescence and mutation among the type-2 lineages will occur quickly, with $K_2$ according to the Ewens sampling formula, coalescence among the type-1 lineages will follow the Kingman coalescent, and $\tau_{n_1}\approx 2\alpha/{\left(\theta_1\theta_2\right)}$. }
    \label{fig:scenario1}
\end{figure}
\FloatBarrier

\begin{thm}\label{T:scenario1}
Suppose $(n_1,n_2)\in\mathbb{N}^2$
 is fixed. Then  under $\P_{\bf n}$, as $\alpha\to\infty$,
\begin{itemize}
\item[(i)] $\sup_{t\in[0,T]} |1-p_t| \to 0$
in probability, for  any $T\in(0,\infty)$; and
\item[(ii)] the triplet
$(K_1,\,K_2,\,s_{n_2})$ converges in distribution to $\left(1,\,\sum_{k=1}^{n_2}\zeta_k,\,0\right)$, where $\{\zeta_k\}$ are independent Bernoulli variables taking values in $\{0,1\}$ and having means $\frac{\theta_2}{\theta_2+k-1}$.
\item[(iii)] $\limsup_{\alpha\to\infty}\tau_{n_1-1}$ is stochastically dominated by the height of the Kingman coalescent with $n_1+n_2$ leaves.
\end{itemize}
\end{thm}

\begin{proof}
By \eqref{eq:sde}, the process $q \coleq 1-p$ solves the  stochastic differential equation
\begin{linenomath*}
\begin{equation}\label{WFDiffusionq1}
dq_t= \sqrt{q_t(1-q_t)}\,dW_t - \frac{\theta_1}{2}q_t +\frac{\theta_2}{2} (1-q_t) -\frac{\alpha}{2} q_t(1-q_t),\quad t\geq 0.
\end{equation}
\end{linenomath*}

Fix any $\epsilon\in(0,1)$. 
We shall to show that $\P_{\bf n}(\sup_{t\in [0,T]}q_t >2\epsilon)\to 0$ as $\alpha\to\infty$.  

By the comparison principle \cite[Proposition 2.18 in Chap. 5]{karatzas1991brownian}, 
we can replace the process $q$ by another process $\widehat{q}$ that solves
\begin{linenomath*}
\begin{align}
d\widehat{q}_t=&\,  \sqrt{\widehat{q}_t(1-\widehat{q}_t)}\,dW_t  + \left[\frac{\theta_2}{2}  -\frac{\alpha}{2} \widehat{q}_t(1-\widehat{q}_t)\right]\,dt
\end{align}
\end{linenomath*}
with an initial condition $\widehat{q}_0$ that is equal in distribution to $q_0$.
By Girsanov's theorem, we can further take away the constant drift $\frac{\theta_2}{2}dt$. That is, it suffices to show that there exists a probability space $(\Omega, \mathcal{F} ,\P)$ on which 
$\P(\sup_{t\in [0,T]}\widehat{Q}_t >2\epsilon)\to 0$ as $\alpha\to\infty$, where 
the  process $\widehat{Q}$ solves
\begin{linenomath*}
\begin{align}
d\widehat{Q}_t=&\,\sqrt{\widehat{Q}_t(1-\widehat{Q}_t)}\,dW_t   -\frac{\alpha}{2} \widehat{Q}_t(1-\widehat{Q}_t)\,dt
\end{align}
\end{linenomath*}
with an initial condition $\widehat{Q}_0$ that is equal in distribution to $q_0$. 
The initial frequency $q_0\to 0$ in probability under $\P_{\bf n}$, by Lemma \ref{L:Convphi}(i).
Hence it suffices to show that
\begin{linenomath*}
\begin{equation}
\P\left(\sup_{t\in [0,T]}\left\{ \int_0^t \sqrt{\widehat{Q}_s(1-\widehat{Q}_s)}\,dW_s   -\frac{\alpha}{2} \int_0^t \widehat{Q}_s(1-\widehat{Q}_s)\,ds \right\} >\epsilon\right)\,\to \,0 \quad\text{ as }\alpha\to \infty.
\end{equation}
\end{linenomath*}
This is true by the time-change representation of the martingale 
$M_t\coleq\int_0^t \sqrt{\widehat{Q}_s(1-\widehat{Q}_s)}\,dW_s$ \cite[Theorem 4.6 in Chap. 3]{karatzas1991brownian} and the fact that
\begin{linenomath*}
\begin{align*}
&\P\left(\sup_{t\in [0,T]}\left\{  B_{\langle M\rangle_t}  -\frac{\alpha}{2}\langle M\rangle_t\right\} >\epsilon\right)\\
=&\,\P\left(\sup_{r\in [0,\,\langle M\rangle_T]}\left\{  B_{r}  -\frac{\alpha}{2}r\right\} >\epsilon\right)\\
\leq &\,\P\left(\sup_{r\in [0,\,T/4]}\left\{  B_{r}  -\frac{\alpha}{2}r\right\} >\epsilon\right)\\
=&\,\int_0^{T/4}\frac{1}{\sqrt{2\pi t^3}}\exp{\left\{\frac{-(\epsilon+\frac{\alpha}{2}t)^2}{2t}\right\}}\,dt \to 0 \quad\text{as}\quad \alpha\to\infty,
\end{align*}
\end{linenomath*}
where in the inequality we used the fact that the quadratic variation $\langle M\rangle_T\leq T/4$ almost surely (since $q_t(1-q_t)\leq 1/4$ for all $t\in\R_+$).
Convergence (i) is proved.

Note that the coalescence rate and the mutation rate in \eqref{E:Evo_type1} converge to $a_1(a_1-1)/2$ and 0 respectively as $p\to 1$. In \eqref{E:Evo_type2} both rates converge to infinity but their  ratio converges in such a way that the limiting one-step transition probabilities are
\begin{linenomath*}
\begin{align}\label{E:one-step type2}
(a_2,\ell_2) &\to 
\begin{cases}
(a_2-1,\ell_2) & \textrm{with probability} \quad \frac{a_2-1}{\theta_2+a_2-1}  \\[6pt]
(a_2-1,\ell_2+1) & \textrm{with probability} \quad \frac{\theta_2}{\theta_2+a_2-1}  \\[6pt]
\end{cases}
\end{align}
\end{linenomath*}

Convergence (ii) then follows from (i) and the representations \eqref{Latent_type1a} and \eqref{Latent_type2a}. 

Finally, for part (iii), note that the convergence $s_{n_2}\to 0$ in part (ii) says that the time for type-2 lineages to disappear is negligible. 
Therefore, it follows from part (i) and \eqref{E:Evo_type1} that
the conditional distribution of $\tau'_{n_1+K_2-1}$, given  $K_2$, converges weakly to the height of the Kingman coalescent with $n_1+K_2$ leaves, where 
$\tau'_{n_1+K_2-1}:=\inf\{t\in\R_+:\,n_1(t)=1\}$ is the first time when the process $n_1$ decreases to 1. Part (iii) then follows
Since $\tau_{n_1-1}\leq \tau'_{n_1+K_2-1}$ and $K_2\leq n_2$ by definition.
\end{proof}

In Theorem \ref{T:Age1} below,
we obtain that the mean of the age $\tau_{n_1}$ is about $\frac{2\,\alpha}{\theta_1\theta_2}$.

\begin{thm}[Age of the oldest latent mutation of a favorable allele]\label{T:Age1}
Suppose $(n_1,n_2)\in\mathbb{N}^2$ is fixed. Then  under $\P_{\bf n}$, as $\alpha\to\infty$, $\frac{\tau_{n_1}}{\alpha}$  converges in distribution to an exponential random variable with mean $\frac{2}{\theta_1\theta_2}$.  That is,
\begin{linenomath*}
\begin{equation*}
\frac{\tau_{n_1}}{\alpha} \,\toL \,{\rm Exp}\left(\frac{\theta_1\,\theta_2}{2}\right)\qquad \text{as}\quad \alpha\to\infty.
\end{equation*}
\end{linenomath*}
\end{thm}

\begin{proof}
By part (iii) of Theorem \ref{T:scenario1}, 
it takes $O(1)$ amount of time for the number of lineages of type 1 to decrease to 1, as $\alpha\to\infty$. It remains to consider the time $\tau_{n_1}-\tau_{n_1-1}$ for this single lineage to mutate.
Recall the rate of mutation in \eqref{E:Evo_type1} with $a_1=1$ lineage,
for any $t\in\R_+$, 
\begin{linenomath*}
\begin{align}
\P_{\bf n}\left(\frac{\tau_{n_1}}{\alpha}>t\right) 
\approx&\, \E\left[e^{\frac{-\theta_1}{2}\int_0^{\alpha t} \frac{1-p_s}{p_s}\,ds}\right] \qquad \text{as }\alpha\to\infty. \label{eq:approxeg}
\end{align}
\end{linenomath*}
The exponent inside the expectation is, by the ergodic theorem and using the stationary probability density \eqref{eq:phix} and \eqref{eq:C},
\begin{linenomath*}
\begin{align*}
\frac{-\theta_1\alpha t}{2} \frac{1}{\alpha t}\int_0^{\alpha t} \frac{1-p_s}{p_s}\,ds \approx&\, \frac{-\theta_1\alpha t}{2} \int_0^1\frac{1-x}{x}\,\phi_\alpha(x)\,dx \qquad \text{almost surely, as }\alpha\to\infty\\
=&\,\frac{-\theta_1\alpha t}{2} \int_0^1\, C\, x^{\theta_1 - 2} (1-x)^{\theta_2} e^{\alpha x}\,dx\\
=&\,\frac{-\theta_1\alpha t}{2} \,\frac{\Gamma(\theta_1-1) \Gamma(\theta_2+1) {}_1F_1(\theta_1-1;\theta_1+\theta_2;\alpha)}{\Gamma(\theta_1) \Gamma(\theta_2) {}_1F_1(\theta_1;\theta_1+\theta_2;\alpha)}\\
\approx &\,\frac{-\theta_1\alpha t}{2} \,\frac{\theta_2}{\alpha} \qquad \text{as }\alpha\to\infty, \quad\text{by }\eqref{eq:1F1s2}\\
=&\, \frac{-\theta_1\theta_2 t}{2}.
\end{align*}
\end{linenomath*}
Hence, by \eqref{eq:approxeg}, $\lim_{\alpha\to\infty}\P_{\bf n}\left(\frac{\tau_{n_1}}{\alpha}>t\right) =e^{\frac{-\theta_1\theta_2 t}{2}}$ for all $t\in\R_+$.
The proof is complete.
\end{proof}

In Theorem \ref{T:Age1}, \eqref{eq:approxeg}, as in all proofs in this paper, $A \approx B$ means that $A/B\to 1$ in the limit specified, which is either $\alpha\to\infty$ or $n\to\infty$.  This is equivalent to $A=B+o(1)$ where $B$ converges and $o(1)$ represents terms which tends to $0$ in the limit.

\subsection{Scenario (ii): arbitrary selection, large sample size} \label{sec:bessub2}

The second scenario is when $n_2$ large with $\alpha$ fixed. We deal with this briefly because it is effectively covered by scenario (iii) when $\widetilde{\alpha}=0$.

The conditional genealogy of the $n_1$ type 1 individuals in the sample can be described as follows. 
Events among the type 1 lineages occur quickly under $\P_{\bf n}$ in the sense that $\tau_{n_1}$ is of order $O(1/n)$. However, if we measure time in proportion to $1/n$ coalescent time units and measure frequency on the scale of numbers of copies of alleles, then the $n_1$ type-1 lineages evolve (coalesce and mutate) as in the Ewens sampling formula. In particular, $K_1\approx 1+\sum_{k=2}^{n_1}\xi_k$, where $\{\xi_k\}$ are independent Bernoulli variables taking values in $\{0,1\}$ and having means $\frac{\theta_1}{\theta_1+k-1}$.

The rescaled frequency process for type 1 can be described precisely under the rescaling above by the Feller diffusion with drift:
\begin{linenomath*}
\begin{equation}\label{limitZ_scenario2}
d Z_t=\sqrt{Z_t}\,dW_t +\frac{\theta_1}{2} \,dt,\qquad t\in\R_+,
\end{equation}
\end{linenomath*}
with the initial distribution being the
Gamma random variable ${\rm Gam}(n_1+\theta_1,1)$.
See Figure \ref{fig:scenario2} for an illustration. Remark \ref{Rk:scenario_ii_iii} below explains how this is a special case of scenario (iii), with $\widetilde{\alpha}=0$.

Equation \eqref{limitZ_scenario2}  (also \eqref{limitZ} below) 
is a Cox-Ingersoll-Ross (CIR) model for interest rates in financial mathematics. It has several other names including  the Feller process and the square-root process \citep{dufresne2001integrated}. It
has a unique strong solution. This equation is not explicitly solvable, but its transition density is explicitly known \citep{vanyolos2014probability} and its moments and distributions have been intensively studied.

\subsection{Scenario (iii): strong selection, large sample size} \label{sec:bessub3}

The third scenario is when both $\lvert\alpha\rvert$ and $n_2$ large with $\widetilde{\alpha} = \alpha/n_2$ fixed.  Lemma \ref{L:Asymp_Initial} implies that 
\begin{linenomath*}
\begin{align}\label{postMean2}
\E_{\bf n}[p_0] &\approx
\begin{cases}
\frac{\theta_1+n_1}{1-\widetilde{\alpha}}\,\frac{1}{n} & \textrm{when } \quad  \widetilde{\alpha}\in (-\infty,1) \\
1-1/\widetilde{\alpha} & \textrm{when } \quad  \widetilde{\alpha}\in (1,\infty) 
\end{cases}
\quad \text{as }n\to\infty.
\end{align}
\end{linenomath*}
Therefore, it makes sense under this scenario to consider two cases: $\widetilde{\alpha}\in (-\infty, 1)$ and $\widetilde{\alpha} \in(1,\infty)$.

\subsubsection{Case \mathinhead{\widetilde{\alpha}\in (-\infty,1)}{(iii)a}}\label{sec:(iii)a}

In this case, under $\P_{\bf n}$ and as $n\to\infty$, we have that $p_0= O(1/n)$ by Lemma \ref{L:Asymp_Initial}. 
The genealogy of the $n_1$ type 1 lineages are the same as that in scenario (ii); see Figure \ref{fig:scenario2}. This description is justified by Theorems \ref{T:limitZ}-\ref{T:JointConvergence} below.

\FloatBarrier
\begin{figure}[h!]
    \centering    
    \includegraphics[scale=0.6]{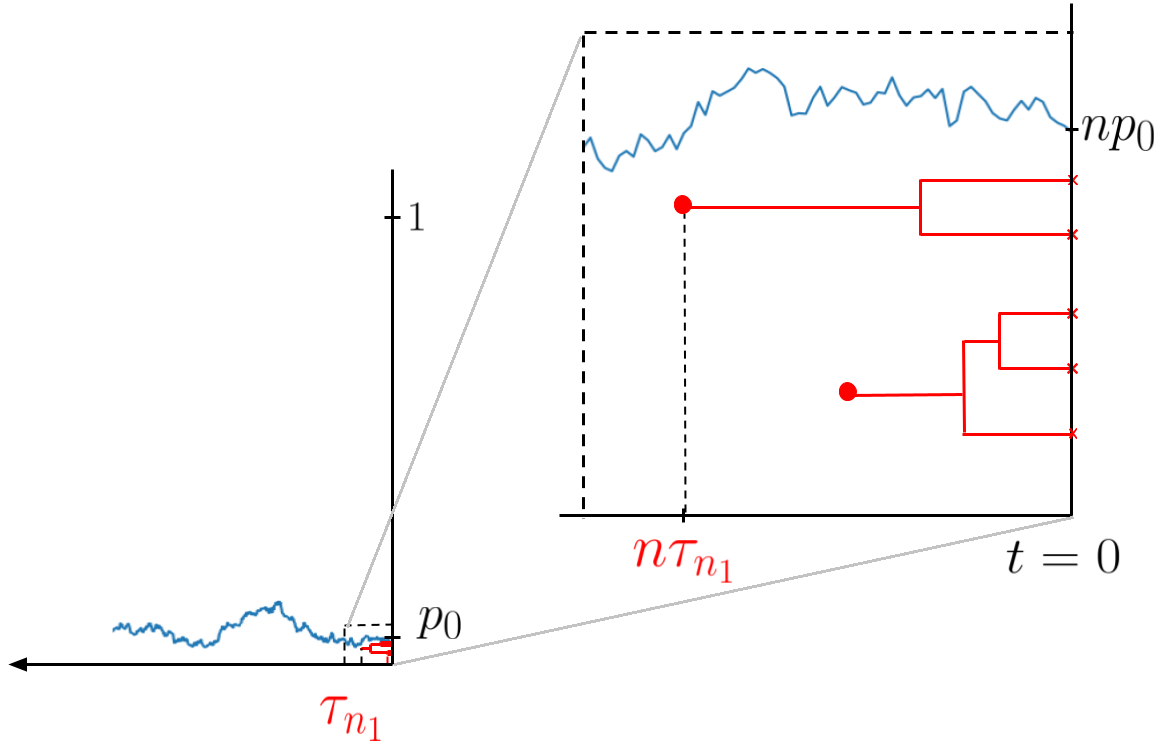}
    \caption{Conditional genealogy of a sample with observed  frequencies $(n_1,n_2)$ at the present time $t=0$, 
    where $n_1=5$ and $n_2$ is large, and $\alpha=\widetilde{\alpha}n_2$ for a constant $\widetilde{\alpha}\in (-\infty,1)$. The $n_2$ samples are not shown.
    In this figure, $K_1=2$ and the two red bullets are mutation events from type 1 to type 2. In
    scenario (iii), $K_1$ is distributed like the number of alleles in the Ewens sampling formula, and the timing of the type-1 events are small (of order $O(1/n)$ on the coalescent time scale). The rescaled process $\big(np_{\frac{t}{n}}\big)_{t\in \R_+}$ is well approximated by the diffusion process \eqref{limitZ} with initial distribution ${\rm Gam}(n_1+\theta_1,1-\widetilde{\alpha})$.
    }
    \label{fig:scenario2}
\end{figure}
\FloatBarrier

Let $(Z_t)_{t\in\R_+}$ be the $\R_+$-valued process  that has initial state $Z_0\sim {\rm Gamma}(n_1+\theta_1,1-\widetilde{\alpha})$ and solves the stochastic differential equation
\begin{linenomath*}  
\begin{equation}\label{limitZ}
d Z_t=\sqrt{Z_t}\,dW_t +\frac{1}{2}  (\theta_1+\widetilde{\alpha}\,Z_t)\,dt,\qquad t\in\R_+,
\end{equation}
\end{linenomath*}
where $W$ is the Wiener process. 

\begin{thm}[Convergence of rescaled genealogy]\label{T:limitZ}
Suppose  $\widetilde{\alpha} = \alpha/n_2\in (-\infty,1)$ is fixed.
As $n\to\infty$, the process $\left(np_{\frac{t}{n}},\,n^{obs}_1(\frac{t}{n}),\,L_1(\frac{t}{n}) \right)_{t\in \R_+}$   converges in distribution under $\P_{\bf n}$, 
in the Skorohod space $D(\R_+,\,\R_+\times\Z_+\times \Z_+)$,
to a Markov process  $\left(Z_t,\,\widetilde{n}_1(t),\,\widetilde{L}_1(t)\right)_{t\in\R_+}$
with state space $\R_+\times \{0,1,\cdots,n_1\}\times \{0,1,\cdots,n_1\}$ described as follows: 
\begin{itemize}
\item[(i)] $\,(Z_t)_{t\in\R_+}$ is a solution to \eqref{limitZ} with initial state $Z_0\sim {\rm Gamma}(n_1+\theta_1,1-\widetilde{\alpha})$. In particular, its transition kernel does not depend on $(\widetilde{n}_1,\,\widetilde{L}_1)$.
\item[(ii)] Suppose the current state is $(z,a_1,\ell_1)$. Then $(\widetilde{n}_1,\,\widetilde{L}_1)$ evolves as
\begin{linenomath*}
\begin{align}\label{E:Evo_type1_tb}
(a_1,\ell_1) &\to 
\begin{cases}
(a_1-1,\ell_1) & \textrm{at rate} \quad \frac{1}{z}\binom{a_1}{2}  \quad \quad \text{ coalescence of type 1} \\[6pt]
(a_1-1,\ell_1+1) & \textrm{at rate} \quad \frac{1}{z}\,a_1\,\frac{\theta_1}{2} \qquad \text{ mutation of 1 to 2}\\[6pt]
\end{cases}
\end{align}
\end{linenomath*}
\end{itemize}
\end{thm}

\begin{proof}
Let $Y_t\coleq np_{t/n}$.  By Lemma \ref{L:Asymp_Initial}, under $\P_{\bf n}$ we have  $Y_0=np_0\to  {\rm Gam}(n_1+\theta_1,1-\widetilde{\alpha})$. By \eqref{eq:sde},
\begin{linenomath*}
\begin{align}
Y_t-Y_0=&\,n\,(p_{t/n}-p_{0}) \notag\\
=&\,n\int_0^{t/n}\sqrt{p_s(1-p_s)}\,dW_s + \frac{n}{2}\int_0^{t/n}\theta_1(1-p_s)-\theta_2 p_s +\alpha p_s(1-p_s)\,ds \notag\\
\eqd&\,\frac{n}{\sqrt{n}}\int_0^{t}\sqrt{p_{r/n}(1-p_{r/n})}\,dW_r + \frac{n}{2}\int_0^{t}\theta_1(1-p_{r/n})-\theta_2 p_{r/n} +\alpha p_{r/n}(1-p_{r/n})\,\frac{dr}{n} \notag\\
=&\,\sqrt{n}\int_0^t\sqrt{\frac{Y_r}{n}\left(1-\frac{Y_r}{n}\right)}\,dW_r\,+\,
 \frac{1}{2}\int_0^{t}\theta_1\left(1-\frac{Y_r}{n}\right)-\theta_2 \frac{Y_r}{n} +\alpha \frac{Y_r}{n}\left(1-\frac{Y_r}{n}\right)\,dr \notag\\
 =&\,\int_0^t\sqrt{Y_r\left(1-\frac{Y_r}{n}\right)}\,dW_r\,+\,
 \frac{1}{2}\int_0^{t}\theta_1\left(1-\frac{Y_r}{n}\right)-\theta_2 \frac{Y_r}{n} + \frac{\alpha}{n} \,Y_r\left(1-\frac{Y_r}{n}\right)\,dr, \label{E:limitY}
\end{align}
\end{linenomath*}
where in the third line above we used the fact that the processes
$\left(\int_0^{t/n}f(s)\,dW_s\right)_{t\in \R_+}$ and $\left( \frac{1}{\sqrt{n}}\int_0^t f(r/n)dW_{r}\right)_{t\in \R_+}$ are equal in distribution, where $f(s)=\sqrt{p_s(1-p_s)}$.

Using \eqref{E:limitY}, the fact $\sup_{n}\E[Y_0^2]<\infty$ and the assumption $\widetilde{\alpha} = \alpha/n_2\in\R$ is fixed, we can check by Gronwall's inequality that $\limsup_{n\to\infty}\E_{\bf n}[\sup_{t\in[0,T]}Y_t]<\infty$ for all $T>0$. 
Now, note that equation \eqref{limitZ} is the same as \eqref{E:limitY} after we get rid of the terms $\frac{Y_r}{n}$  and replace $\frac{\alpha}{n}$
by $\widetilde{\alpha}$.
As $n\to\infty$, the process $(Y_t)_{t\in[0,T]}$ converges in distribution under $\P_{\bf n}$ to a process $(Z_t)_{t\in[0,T]}$ with initial state $Z_0 \sim  {\rm Gam}(n_1+\theta_1,1-\widetilde{\alpha})$ and solving \eqref{limitZ}. 

Using \eqref{E:Evo_type1}, the desired weak convergence 
in the Skorohod space $D(\R_+,\,\R_+\times\Z_+\times \Z_+)$
can be checked using a standard compactness argument as in \citet[Chap. 2]{billingsley2013convergence} or \citet[Chap. 3]{ethier2009markov}. 
That is, we first show that the family  is relatively compact: any subsequence has a further subsequence that converges in distribution as $n\to\infty$. This can be done using the Prohorov's theorem. Next, we identify that any subsequential limit is equal in distribution to the process $(Z,\,\widetilde{n}_1,\,\widetilde{L}_1)$, by showing that they solve the same martingale problem. 
\end{proof}

By Theorem \ref{T:limitZ}, the jump times of the process
$\left(n^{obs}_1(\frac{t}{n})\right)_{t\in \R_+}$ converge to those of the process $\widetilde{n}_1$ as $n\to\infty$. See, for instance, Proposition 5.3 in \citet[Chap. 3]{ethier2009markov}. We give a stronger statement and an explicit proof in Theorem \ref{T:JointConvergence} below. Theorem \ref{T:JointConvergence} also implies that  when $\widetilde{\alpha}\in(-\infty,1)$,   
the total number of latent mutations for type 1 is predicted by the Ewens sampling formula, as $n\to\infty$.
Let $\widetilde{\tau}_1< \widetilde{\tau}_2<\cdots<  \widetilde{\tau}_{n_1}$ be the jump times of the process $\widetilde{n}_1$ in Theorem \ref{T:limitZ}, at each of which the process decreases by 1.

\begin{thm}[Timing of events and number of mutations] \label{T:JointConvergence}
Suppose $\widetilde{\alpha} = \alpha/n_2\in (-\infty,1)$ fixed.
Then as $n\to\infty$, 
\begin{itemize}
\item[(i)]  the random vector $\left(n_2\,\tau_i,\;n_2\,p_{\tau_i}\right)_{i=1}^{n_1}$ under $\P_{\bf n}$
 converges in distribution to $\left(\widetilde{\tau}_{i},\;Z_{\widetilde{\tau}_{i}}\right)_{i=1}^{n_1}$.
\item[(ii)] $K_1$ converges in distribution under $\P_{\bf n}$ to $1+\sum_{k=2}^{n_1}\xi_k$, where $\{\xi_k\}$ are independent Bernoulli variables taking values in $\{0,1\}$ and having means $\frac{\theta_1}{\theta_1+k-1}$.
\end{itemize}
\end{thm}

\begin{proof}
For part (i), we first give a more explicit description of the jump times $\widetilde{\tau}_1< \widetilde{\tau}_2<\cdots<  \widetilde{\tau}_{n_1}$, in terms of the function
\begin{linenomath*}
\begin{equation*}
\widetilde{\lambda}_{a_1}(z)\coleq\frac{a_1}{2z}(\theta_1+a_1-1)
\end{equation*}
\end{linenomath*}
that comes from \eqref{E:Evo_type1_tb} in Theorem \ref{T:limitZ}.
At the first jump time $\widetilde{\tau}_1$, the process $\widetilde{n}_1$ decreases from $n_1$ to $n_1-1$.
Thus $\widetilde{\tau}_1$ is the first jump time of a Poisson process with time inhomogeneous rate $\left(\widetilde{\lambda}_{n_1}(Z_t)\right)_{t\in\R_+}$, given the trajectory $(Z_t)_{t\in\R_+}$. Hence,
\begin{linenomath*}
\begin{equation}\label{ttau1}
\P(\widetilde{\tau}_1>t)= \E\left[e^{-\int_0^t \widetilde{\lambda}_{n_1}(Z_s)\,ds}\right], \qquad t\in\R_+.
\end{equation}
\end{linenomath*}
Given $(\widetilde{\tau}_1,\,Z_{\widetilde{\tau}_1})$, the difference $ \widetilde{\tau_2}- \widetilde{\tau_1}$ is 
 the first jump time of an independent Poisson process with time inhomogeneous rate $\left(\widetilde{\lambda}_{n_1-1}(Z_{t+\widetilde{\tau}_1})\right)_{t\in\R_+}$. 
Given $(\widetilde{\tau}_2,\,Z_{\widetilde{\tau}_2})$, the difference $ \widetilde{\tau_3}- \widetilde{\tau_2}$ is  the first jump time of an independent Poisson process with time inhomogeneous rate $\left(\widetilde{\lambda}_{n_1-2}(Z_{t+\widetilde{\tau}_2})\right)_{t\in\R_+}$;  and so on. Finally,
 given $(\widetilde{\tau}_{n_1-1},\,Z_{\widetilde{\tau}_{n_1-1}})$, the difference $ \widetilde{\tau}_{n_1}- \widetilde{\tau}_{n_1-1}$ is 
 the first jump time of an independent Poisson process with time inhomogeneous rate $\left(\widetilde{\lambda}_{1}(Z_{t+\widetilde{\tau}_{n_1-1}})\right)_{t\in\R_+}$.  

Using the total rate of type-$1$ events, $\lambda_{a_1}(p)$ defined in \eqref{taurate}, and Theorem \ref{T:limitZ}, as $n\to\infty$ we have
\begin{linenomath*}
\begin{align*}
\int_0^{t/n}\lambda_{a_1}(p_s)\,ds =\int_0^t \frac{\lambda_{a_1}(p_{s/n})}{n}\,ds =&\,\int_0^t \frac{a_1}{n\,2p_{s/n}}\big(a_1-1+(1-p_{s/n})\theta_1\big)\,ds \\
\to&\, \int_0^t \frac{a_1}{2\,Z_s}(\theta_1+a_1-1) \,ds = \int_0^t\widetilde{\lambda}_{a_1}(Z_s)\,ds.
\end{align*}
\end{linenomath*}
Hence $\P_{\bf n}(n_2\tau_1>t)\to \P(\widetilde{\tau}_1>t)$ for all $t\geq 0$, by \eqref{ttau1}. 
Combining with Theorem \ref{T:limitZ},
we have that   $n_2\left(\tau_1,\;p_{\tau_1}\right)$ under $\P_{\bf n}$ converges in distribution to $\left(\widetilde{\tau}_{1},\;Z_{\widetilde{\tau}_{1}}\right)$ as $n_2\to\infty$.

Applying the strong Markov property of the process $Z$ at $\widetilde{\tau}_{1}$ and that of the process $p$ at $\tau_1$, we can similarly show that
$n_2\left(\tau_1,\tau_2-\tau_1,\;p_{\tau_1}, p_{\tau_2}-p_{\tau_1}\right)$ under $\P_{\bf n}$ converges to $\left(\widetilde{\tau}_{1},  \widetilde{\tau}_{2}-\widetilde{\tau}_{1},\;Z_{\widetilde{\tau}_{1}}, Z_{\widetilde{\tau}_{2}}-Z_{\widetilde{\tau}_{1}}\right)$ in distribution. Continuing in the same way, we obtain that $n_2\left(\tau_i-\tau_{i-1},\;p_{\tau_i}-p_{\tau_{i-1}}\right)_{i=1}^{n_1}$ under $\P_{\bf n}$ converges in distribution to $\left(\widetilde{\tau}_{i}-\widetilde{\tau}_{i-1},\;Z_{\widetilde{\tau}_{i}}-Z_{\widetilde{\tau}_{i-1}}\right)_{i=1}^{n_1}$, where $\tau_0=\widetilde{\tau}_{0}=0$. The desired convergence in part (i) then follows.

We now prove part (ii). 
the vector $(p_{\tau_i})_{i=1}^{n_1-1}$ converges in probability to the zero vector in $\R^{n_1-1}$ as $n\to\infty$, by Theorem \ref{T:limitZ}. This implies that
\begin{linenomath*}
\begin{equation*}
\left(h_{k}(p_{\tau_{n_1-k+1}})\right)_{k=2}^{n_1} \to \left(\frac{\theta_1}{\theta_1+k-1}\right)_{k=2}^{n_1} \in \R^{n_1-1},
\end{equation*}
\end{linenomath*}
where we recall  $h_{k}(p)\coleq\frac{(1-p)\theta_1}{(1-p)\theta_1+k-1 }$ defined in \eqref{tauprob}. Hence the following weak convergence in $\R^{n_1-1}$ holds:
\begin{linenomath*}
\begin{equation*}
\left(\xi_{k}(p_{\tau_{n_1-k+1}})\right)_{k=2}^{n_1} \toL \left(\xi_k\right)_{k=2}^{n_1}.
\end{equation*}
\end{linenomath*}
The proof of part (ii) is complete by \eqref{Latent_type1a}.
\end{proof}

In \citet[Appendix]{WakeleyEtAl2023}, we showed that for the case $\alpha=0$ (no selection) and $n_1>1$, the jump times 
$\tau_1<\tau_2<\cdots< \tau_{n_1}$ of  $n^{obs}_1$
are of order $1/n$ and the re-scaled vector $(n\tau_i)_{i=1}^{n_1}$ converges in distribution. 
Theorem \ref{T:JointConvergence} therefore generalizes the latter convergence in the presence of selection, in scenario (ii) and in the case $\widetilde{\alpha}\in (-\infty,1)$ within scenario (iii). This can further be generalized to time-varying populations, as we shall show below.
By equation (18) in \citet{WakeleyEtAl2023},  $\E_{\bf n}[\tau_1]\approx \frac{2\log n_2}{n_2}$  and $\E_{\bf n}[\widetilde{\tau}_{1}]=\infty$  when $n_1=1$ and $\alpha=0$.

\begin{remark}\rm \label{Rk:scenario_ii_iii}
Theorems \ref{T:limitZ}-\ref{T:JointConvergence}  remain valid if ``$\widetilde{\alpha} = \alpha/n_2\in (-\infty,1)$ is fixed" is replaced by ``$\alpha=\widetilde{\alpha}n_2+c$ where $\widetilde{\alpha},\,c\in\R$ are fixed". In particular, these results hold for scenario (ii).
\end{remark}

\subsubsection{Case \mathinhead{\widetilde{\alpha}\in (1,\infty)}{(iii)b}}\label{sec:(iii)b}

In this case, under $\P_{\bf n}$ and as $n\to\infty$, we have that
$p_0\to 1-1/\widetilde{\alpha}>0$ by \eqref{Dirac_initial}. 
The process $(p_t)_{t\in\R_+}$ increases very quickly and stays close to 1. 
As a result, the conditional ancestral process for the $n_1$ type 1 samples has two parts with different timescales.
First it
coalesces only as the Kingman coalescent (without mutation) until there is only one lineage. Then it takes a very long time (about $n\frac{2\,\widetilde{\alpha}}{\theta_1\theta_2} \approx \frac{2\,\alpha}{\theta_1\theta_2}$) for the single latent mutation to occur. In particular, $K_1\approx 1$.

This description is justified by Theorems \ref{T:tildealpha>1}-\ref{T:JointConvergence2}. See Figure \ref{fig:scenario3b} for an illustration.

\FloatBarrier
\begin{figure}[h!]
    \centering
    \includegraphics[scale=0.6]{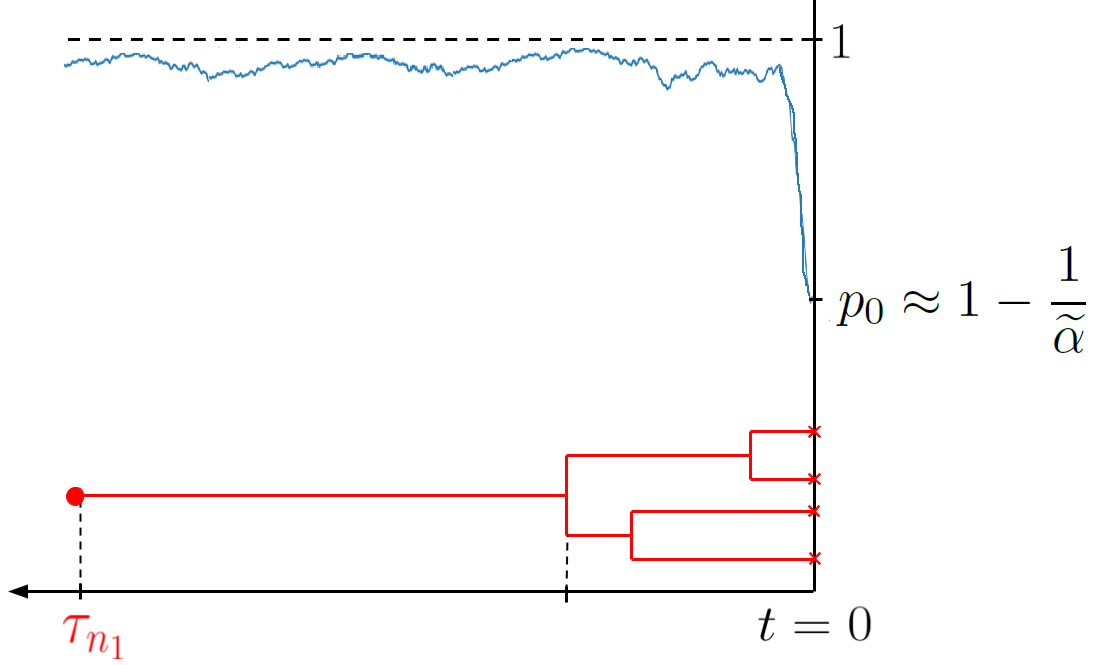}
    \caption{
    Conditional genealogy of a sample with observed  frequencies $(n_1,n_2)$ at the present time $t=0$, where $n_1=4$ and $n_2$ is large, and $\alpha=\widetilde{\alpha}n_2$ for a constant $\widetilde{\alpha}\in (1,\infty)$. The $n_2$ samples are not shown. This scenario is reminiscent of scenario (i) if we focus on the genealogy of only the $A_1$ lineages.
    These lineages first coalesce as the Kingman coalescent (without mutation) until there is only one lineage, which take $O(1)$ amount of time. Then it takes $O(n_2)$ amount of time for the single latent mutation to occur at time $\tau_{n_1}$. 
    }
    \label{fig:scenario3b}
\end{figure}
\FloatBarrier

\begin{thm}
\label{T:tildealpha>1}
Suppose $\widetilde{\alpha} = \alpha/n_2\in(1,\infty)$ is fixed. 
As $n\to\infty$,  under $\P_{\bf n}$,
\begin{itemize}
\item[(i)] $\sup_{t\in[S,T]} |1-p_t| \to 0$
in probability, for  any $0<S<T<\infty$; and
\item[(ii)] For any $T>0$, the process $(n^{obs}_1(t),\,L_1(t))_{t\in [0,T]}$ converges in distribution 
to a process $( \widetilde{n}^{obs}(t),\,0)_{t\in [0,T]}$, where $\widetilde{n}^{obs}$ is a pure death process with jump rate $\binom{k}{2}=\frac{k(k-1)}{2}$ from  $k$ to $k-1$.
\item[(iii)]  $K_1\to 1$ in probability.
\end{itemize}
\end{thm}

\begin{proof}
We first observe that the process $p$ gets close to 1 quickly, when $\widetilde{\alpha}\in (1,\infty)$, in the sense that for any $\epsilon\in(0,1)$,
\begin{linenomath*}
\begin{equation*}
\lim_{n\to\infty}\P_{\bf n}\left(\tau_{1-\epsilon}>S\right)=\lim_{n\to\infty}\P_{\bf n}\left(\sup_{t\in[0,S]}p_{t}<1-\epsilon\right)\,=\,0,
\end{equation*}
\end{linenomath*}
where $\tau_{1-\epsilon}\coleq\inf\{t\in\R_+:\,p_t>1-\epsilon\}$ is the first time $p$ hits a value above $1-\epsilon$. This is true because $p_0\to 1-1/\widetilde{\alpha}\in (0,1)$ in probability, so that the growth term $\widetilde{\alpha} n_2 \,p_t(1-p_t)$ is large at least when $t>0$ is small.
Next, suppose the process $p$ starts at $1-\epsilon$ (i.e. the process $q=1-p$ starts at $\epsilon$), we show that the exit time of the process $q$ out of the interval $[0,2\epsilon)$ is longer than $T$ with probability tending to 1, as $n\to\infty$. More precisely,
\begin{linenomath*}
\begin{align*}
&\P_{1-\epsilon}\left(\inf_{t\in[0,\,T]} p_t < 1-2\epsilon\right)=\,\P_{\epsilon}\left(\sup_{t\in[0,T]} q_t > 2\epsilon\right) 
\end{align*}
\end{linenomath*}
which tends to 0 as $n_2\to\infty$, as in the proof of Theorem \ref{T:scenario1}(i).

From these two estimates and the strong Markov property of the process $p$, we have that for any $\epsilon\in(0,1)$, 
\begin{linenomath*}
\begin{align*}
&\P_{\bf n}\left(\sup_{t\in[S,T]} |1-p_t| > 2\epsilon\right)\\
\leq&\, \P_{\bf n}\left(\sup_{t\in[S,T]} |1-p_t| > 2\epsilon,\,\tau_{1-\epsilon}\leq S\right) +\P_{\bf n}\left(\tau_{1-\epsilon}>S\right)\\
=&\,  \E_{\bf n}\left[ 1_{\{\tau_{1-\epsilon}\leq S\}} \,\P_{1-\epsilon}\left(\inf_{t\in[S-\,\tau_{1-\epsilon},\,T-\,\tau_{1-\epsilon}]} p_t < 1-2\epsilon\right) \right]+\P_{\bf n}\left(\tau_{1-\epsilon}>S\right)\\
\leq &\, \P_{1-\epsilon}\left(\inf_{t\in[0,\,T]} p_t < 1-2\epsilon\right) +\P_{\bf n}\left(\tau_{1-\epsilon}>S\right)\to 0 \qquad\text{as }n\to\infty.
\end{align*}
\end{linenomath*}
The proof of part (i) is complete.

By part (i) and \eqref{taurate}, the times for the type 1 events are of order $O(1)$ and   $h_k(p_t)\to 0$ where $h_k(p)$ is defined in \eqref{tauprob}. Hence parts (ii) and (iii) follow.
\end{proof}

Now we consider the second part of the genealogy,  when there is a single lineage left (i.e. during $\tau_{n_1-1}$ and $\tau_{n_1}$).

To estimate the age $\tau_{n_1}$ of the single latent mutation, we can ignore the $n_1-1$ jump times of the process $\widetilde{n}^{obs}$ (since they are of order 1 by Theorem \ref{T:tildealpha>1}). The frequency of type 1 is tightly regulated in the sense that it is  close to 1 in the sense of Theorem \ref{T:tildealpha>1}(i). However, we need to know ``how close it is to 1'' in order to get an estimate of  $\tau_{n_1}$, because simply setting $p=1$ in $\frac{1-p}{2p}\theta_1$ will give us zero.

Theorem \ref{T:JointConvergence2} below is analogous to Theorem \ref{T:Age1}.
We obtain that the mean of the age $\tau_{n_1}$ is about $n\frac{2\,\widetilde{\alpha}}{\theta_1\theta_2} \approx \frac{2\,\alpha}{\theta_1\theta_2}$ when it is larger than $\frac{2n}{\theta_1\theta_2}$ and $n$ is large.
\begin{thm}[Age of the unique latent mutation]\label{T:JointConvergence2}
Suppose  $\widetilde{\alpha} = \alpha/n_2\in(1,\infty)$ is fixed.
As $n\to\infty$, $\frac{\tau_{n_1}}{n}$  converges in distribution under $\P_{\bf n}$ to an exponential random variable with mean $\frac{2\,\widetilde{\alpha}}{\theta_1\theta_2}$.
That is,
\begin{linenomath*}
\begin{equation*}
 \frac{\tau_{n_1}}{n} \,\toL \,{\rm Exp}\left(\frac{\theta_1\,\theta_2}{2\,\widetilde{\alpha}}\right)\qquad \text{as}\quad n\to\infty.
\end{equation*}
\end{linenomath*}
\end{thm}

\begin{proof}
By Theorem \ref{T:tildealpha>1} (ii), 
for any $t\in\R_+$, 
\begin{linenomath*}
\begin{align}
\P_{\bf n}\left(\frac{\tau_{n_1}}{n}>t\right) 
\approx &\, \E\left[\exp{\left\{\frac{-\theta_1\,t}{2n^2t}\int_0^{n^2t} \frac{n(1-p_{r/n})}{p_{r/n}}\,dr\right\}}\right] \qquad \text{as } n\to\infty.\notag
\end{align}
\end{linenomath*}
The rest follows exactly as  the proof of Theorem \ref{T:Age1}.
\end{proof}

\subsection{Time-varying population size}\label{S:varying}

For a population with time-varying size  $\rho(t)N$ at forward time $t$ where $\rho$ is a non-constant  function, neither the Moran process nor its diffusion approximation  possess a stationary distribution. However,  the random background approach of \citet{BartonEtAl2004} can be generalized to this setting by considering the time-reversed frequency process. 

Our main message in this section is that the limiting results in scenarios (i) and (ii) are robust against continuously-changing population sizes and the initial distribution $\mu_0$ of the initial (ancient) frequency $X_0$. Roughly speaking, in scenario (ii) as $n_2 \sim n \to \infty$, events among the finite-count $A_1$ alleles in the sample are so sped up that the population size will have hardly changed by the time all their coalescent and latent mutation events have occurred.  The same is true for events among the finite-count $A_2$ alleles in scenario (i) as $\alpha\to\infty$.  Events among the finite-count $A_1$ alleles occur more slowly in scenario (i), but the rate of latent mutations among them remains exceedingly small.  The limiting result for scenario (iii) is more subtle. It depends on the large deviation behavior of the present day frequency as $n_2\to\infty$. 

Let $T>0$ be the present.  For comparison with our results for constant population size, we keep the same definitions of $\theta_1$, $\theta_2$ and $\alpha$ and we set $\rho(T) = 1$. Thus, $N$ is the population size at the present time $T$, and $\theta_1$, $\theta_2$ and $\alpha$ are the present-day values of these variables.  The corresponding values at some other time $t$ are $N\rho(t)$, $\theta_1\rho(t)$, $\theta_2\rho(t)$ and $\alpha\rho(t)$.  The demographic function $\rho$ could for example represent exponential population growth, in which case $\rho(t)=\rho(0)e^{\beta t}$ for some positive constant $\beta$.  This model was used in \citet{WakeleyEtAl2023} to illustrate the effects of rapid growth on neutral rare variation in humans.  Here we allow that $\rho(t)$ is piecewise continuous.  As will become clear, the key feature of $\rho$ for our results is that it is continuous at $T$.

Since the random background approach of \citet{BartonEtAl2004} was formulated based on the lineage dynamics of the Moran model, we begin by describing the diffusion process arising from a Moran model with time-varying population size. 

\begin{lemma}[Diffusion limit for time-varying Moran model]\label{L:MoranVary}
Let $\rho:\,\R_+\to(0,\infty)$ be a   
piecewise continous function with finitely many jumps, and $N$ be a positive integer.
Consider the discrete-time Moran process in which, at step $k=[N(N-1)t/2]$,  the total population size is $[\rho(t)N]$ and  $N$ is replaced by $\rho(t)N$ in the one-step transition probabilities \eqref{eq:pjplusone}-\eqref{eq:pjminusone}. Suppose $u_1=\theta_1/N$, $u_2=\theta_2/N$ and $s=-\alpha/N$. Then  as $N\to\infty$,
the relative frequency of $A_1$  at step $[N(N-1)t/2]$  converges in distribution to $X_t$ solving
\begin{linenomath*}
\begin{equation}
dX_t = \left[\frac{\theta_1}{2\rho(t)} (1-X_t) - \frac{\theta_2}{2\rho(t)} X_t + \frac{\alpha}{2\rho(t)} X_t (1-X_t) \right]dt + \sqrt{\frac{X_t (1-X_t)}{\rho^2(t)}}\, dW_t, \label{eq:sde_vary}
\end{equation}
\end{linenomath*} 
where $W_t$ is the Wiener process, provided that the initial relative frequency converges to $X(0)$.
\end{lemma}

Setting $\rho(t)\equiv 1$ for all $t\in\R_+$, or $\beta=0$ in the exponential growth model, makes \eqref{eq:sde_vary} identical to \eqref{eq:sde}. 
The term $\rho^2(t)$ in the denominator inside the square root comes from the diffusion timescale of the Moran model: for a population of constant size, one unit of time in the diffusion is $N(N-1)/2 \propto N^2$ time steps in the discrete model.  To explain the term $\frac{\alpha}{2\rho(t)}=\frac{-s \rho(t)N}{2\rho^2(t)}$, note that the rate of change of $X_t$ due to selection is proportional to the product of the total size $\rho(t)N$ and the parameter $s$, which is then multiplied by $1/\rho^2(t)$ because the timescale in \eqref{eq:sde_vary} is defined in terms of the present-day population size $N$.  The proof of Lemma \ref{L:MoranVary} is given in \ref{sec:lemmaproofs}.

\begin{remark}[Wright-Fisher model with varying size]\rm
The analogous diffusion process $X^{\rm WF}$ for the discrete Wright-Fisher model with total size $[\rho(t)N]$ in generation $[Nt]$ is \textit{different from} the process $X$ in \eqref{eq:sde_vary}, except in the case $\rho(t)\equiv 1$ for all $t\in \R_+$.  This  diffusion solves the SDE
\begin{linenomath*}
\begin{equation*}
dX^{\rm WF}_t = \left[\frac{\theta_1}{2} (1-X^{\rm WF}_t) - \frac{\theta_2}{2} X^{\rm WF}_t + \frac{\alpha}{2} X^{\rm WF}_t (1-X^{\rm WF}_t) \right]dt + \sqrt{\frac{X^{\rm WF}_t (1-X^{\rm WF}_t)}{\rho(t)}}\, dW_t, 
\end{equation*}
\end{linenomath*} 
which is the adaptation of equation (1) in \citet{SchraiberEtAl2016} to our haploid model of selection and recurrent mutation; see also equation (21) in \citet{evans2007non}.  The generators of $X^{\rm WF}$ and $X$ are related by $\mathcal{A}^{\rm WF}_t=\rho(t)\mathcal{A}_t$ for all $t\in\R_+$.  In other words, the diffusion $X^{\rm WF}$ from the discrete Wright-Fisher model is sped up by the factor $\rho(t)$ at time $t$. 

To compare $X$ and $X^{\rm WF}$, 
we can perform deterministic time-changes to normalize their diffusion coefficients to be the same.
In general, suppose $X$ satisfies the SDE
$dX_t= b(t,X_t) dt+ \sigma(t,X_t)\, dW_t$ and
$Y$ is a time-change of $X$ defined by $Y_r:=X_{\psi(r)}$, where $\psi$ is any fixed continuous and strictly increasing function, then
\begin{linenomath*}
\begin{equation*}
dY_r= b(\psi(r),Y_r)\,\psi'(r)\,dr +\sigma(\psi(r),Y_r)\,\sqrt{\psi'(r)}\,dW_r.
\end{equation*}
\end{linenomath*}
Hence, 
when $X$ is a (weak) solution to \eqref{eq:sde_vary} and $\psi=g^{-1}$ where
$g$ is the unique continuous function such that $g(t)=\int_0^t\frac{1}{\rho^2(s)}ds$, we obtain that $Y_r:=X_{g^{-1}(r)}$ solves
\begin{linenomath*}
\begin{equation}
dY_r =  \rho(g^{-1}(r))\,b(Y_r)\,dr + \sqrt{Y_r (1-Y_r)}\, dW_r, \label{eq:sde_vary_normalized}
\end{equation}
\end{linenomath*} 
where $b(y)=\frac{\theta_1}{2} (1-y) - \frac{\theta_2 }{2} y + \frac{\alpha}{2} y (1-y)$.
Analogously, following \citet{SchraiberEtAl2016}---see their equation (6) and the SDE below it---and taking $f$ such that $f(t)=\int_0^t\frac{1}{\rho(s)}ds$, we find that $Y^{\rm WF}_r:=X^{\rm WF}_{f^{-1}(r)}$ solves
\begin{linenomath*}
\begin{equation}
dY^{\rm WF}_r =  \rho(f^{-1}(r))\,b(Y^{\rm WF}_r)\,dr + \sqrt{Y^{\rm WF}_r (1-Y^{\rm WF}_r)}\, dW_r. \label{eq:WFsde_vary_normalized}
\end{equation}
\end{linenomath*}  
Since $f\neq g$ unless $\rho(t)\equiv 1$ for all $t\in\R_+$, we have that $Y\neq Y^{\rm WF}$, i.e. $X_{g^{-1}(r)}\neq X^{\rm WF}_{f^{-1}(r)}$ in general. Nonetheless, \eqref{eq:sde_vary_normalized} and \eqref{eq:WFsde_vary_normalized} have the same form, the only difference being the way time $r$ in these diffusions is related back to time $t$ in the discrete models.  
\end{remark}

\medskip

Note that the law of the present-time frequency of $A_1$ in the model of Lemma~\ref{L:MoranVary} depends on the distribution $\mu_0$ of the initial frequency $X(0)$. This law is denoted by $\P_{\mu_0}(X_T\in dy)$. 

Suppose a sample of $n$ individuals are picked uniformly at random at the present time $T>0$, i.e.\ when the frequency of $A_1$ is $X_T$, and we know that $n_1$ of them are of type 1 (and $n-n_1$ are of type 2).
Let $\P_{\bf n}=\P_{\bf n,\mu_0}$ be the conditional probability measure given the sample count ${\bf n}=(n_1,n_2)$. We also denote the conditional law of the present frequency
$p_0=X_T$, under $\P_{\bf n}$, by $\mathcal{L}^{{\bf n}}\coleq\P_{\bf n}(X_T\in dy)= \P_{\mu_0}(X_T\in dy\,|\,{\bf n})$. 
Then 
\begin{linenomath*}
\begin{align} 
\mathcal{L}^{{\bf n}} = C_T\, y^{n_1} (1-y)^{n_2}\,\P_{\mu_0}(X_T\in dy), \label{eq:condlaw_T}
\end{align}
where $C_T=\left(\int_0^1 y^{n_1} (1-y)^{n_2} \P_{\mu_0}(X_T\in dy)  \right)^{-1}$ is a normalizing constant.
\end{linenomath*}
This follows from Bayes' theorem, just like \eqref{eq:condphi1} did, but with prior distribution $\P_{\mu_0}(X_T\in dy)$.

Similar to Proposition \ref{prop:BothType}, the conditional ancestral process in the diffusion limit  can be described as follows. This description involves  the backward frequency process   
\begin{linenomath*}
\begin{equation}\label{pt_vary}
p_t \coleq X_{T-t} \qquad \text{for} \quad t\in[0,T]
\end{equation}
\end{linenomath*} 
which is by definition the time-reversal of the process $X$. 

\begin{prop}[Conditional ancestral process]\label{prop:BothType_vary}
Let $T>0$ and $\mu_0\in \mathcal{P}([0,1])$ be fixed, and the demographic function $\rho$ be as in Lemma \ref{L:MoranVary}. 
The process $(p_t,\,n^{obs}_1(t),\,L_1(t),\,n^{obs}_2(t),\,L_2(t))_{t\in [0,T]}$ under $\P_{\bf n}$ is a time-inhomogeneous Markov process with state space $[0,1]\times \{0,1,\cdots,n_1\}^2 \times  \{0,1,\cdots,n_2\}^2$ described as follows: 
\begin{itemize}
\item[(i)] $\,(p_t)_{t\in[0,T]}$, defined by \eqref{pt_vary}, has the law of $(X_{T-t})_{t\in[0,T]}$ under $\P_{{\bf n}}$. In particular, it has initial distribution $\mathcal{L}^{{\bf n}}$ in \eqref{eq:condlaw_T} and it does not depend on the process $(n^{obs}_1,\,L_1,\,n^{obs}_2,\,L_2)$.
\item[(ii)] The process $(n^{obs}_1,\,L_1,\,n^{obs}_2,\,L_2)$ starts at $(n_1,0,n_2,0)$.
When this process is at time $t$ and the current state is $(p,a_1,\ell_1,a_2,\ell_2)$, this process evolves as
\begin{linenomath*}
\begin{align}\label{E:Evo_type1_vary}
(a_1,\ell_1) &\to 
\begin{dcases}
(a_1-1,\ell_1) & \textrm{at rate} \quad \frac{1}{p}\binom{a_1}{2}\frac{1}{\rho^2(T-t)} \quad \quad \text{ coalescence of type 1}^{obs} \\[6pt]
(a_1-1,\ell_1+1) & \textrm{at rate} \quad \frac{1-p}{p}\,a_1\,\frac{\theta_1}{2 \rho(T-t)} \qquad  \text{ mutation of 1}^{obs}\text{ to 2}\\[6pt]
\end{dcases}
\end{align}
\end{linenomath*}
and, independently,
\begin{linenomath*}
\begin{align}\label{E:Evo_type2_vary}
(a_2,\ell_2) &\to 
\begin{dcases}
(a_2-1,\ell_2) & \textrm{at rate} \; \frac{1}{1-p}\binom{a_2}{2}\frac{1}{\rho^2(T-t)}   \quad \quad \text{ coalescence of type 2}^{obs} \\[6pt]
(a_2-1,\ell_2+1) & \textrm{at rate} \quad \frac{p}{1-p}\,a_2\,\frac{\theta_2}{2  \rho(T-t)} \qquad  \text{ mutation of 2}^{obs}\text{ to 1}\\[6pt]
\end{dcases}
\end{align}
\end{linenomath*}
\end{itemize}
\end{prop}

\smallskip

Note the term $\rho(T-t)$ in \eqref{E:Evo_type1_vary}-\eqref{E:Evo_type2_vary} indicates the dependence of the conditional ancestral process on the demographic function. Nonetheless, Proposition \ref{prop:BothType_vary} still gives a description 
 for $K_1$ and $K_2$ in terms of Bernoulli random variables,
like \eqref{Latent_type1a} and \eqref{Latent_type2a} respectively.  
For example, under $\P_{\bf n}$, \eqref{Latent_type1a} still holds but \eqref{xi_K1} needs to be modified. 
Indeed, given $\{(\tau_i,p_{\tau_i})\}_{i=1}^{n_1-1}$, the random variables $\{\xi_{k}(\cdot)\}_{k=2}^{n_1}$ are independent and
\begin{linenomath*}
\begin{align} 
\xi_{n_1-i+1}(p_{\tau_{i}}) &=
\begin{dcases}
0 & \textrm{with probability} \quad \frac{n_1-i}{(1-p_{\tau_{i}})\theta_1\rho(T-\tau_i)+n_1-i }  \\[6pt]
1 & \textrm{with probability} \quad \frac{(1-p_{\tau_{i}})\theta_1\rho(T-\tau_i)}{(1-p_{\tau_{i}})\theta_1\rho(T-\tau_i)+n_1-i } 
\end{dcases}\label{xi_K1_vary}
\end{align}
\end{linenomath*}
which further generalizes $\xi_k(p)$ in \eqref{xi_K1} to include $\rho(t)$.  An analogous description holds for $K_2$.

\begin{remark}\rm\label{Rk:SDE for p}
A more explicit description for the process $(X_{t})_{t\in[0,T]}$ under $\P_{\bf n}$, hence also that for its time-reversal $(p_{t})_{t\in[0,T]}$, can be obtained by Doob's h-transform \citep{Doob1957,Doob2001}. More precisely, we define the function
\begin{linenomath*}
\begin{equation*}
h(t,x)\coleq\P({\bf n}\,|\,X_t=x)=\E\left[\binom{n}{n_1}X_{T}^{n_1}(1-X_{T})^{n_2}\,\Big|\,X_t=x\right].
\end{equation*}
\end{linenomath*}
Then $(X_{t})_{t\in[0,T]}$ under the conditional probability $\P_{\bf n}$ solves the SDE 
\begin{linenomath*}
\begin{equation*}
dX_t=
 \left[b(t,X_t) + \sigma^2(t,X_t)\,\partial_x\log h(t,X_t) \right]dt+ \sigma(t,X_t)\, d\widetilde{W}_t,
\end{equation*}
\end{linenomath*}
where $b(t,x)\coleq \frac{\theta_1}{2 \rho(t)} (1-x) - \frac{\theta_2}{2 \rho(t)}x + \frac{\alpha}{2 \rho(t)} x (1-x)$  and $\sigma(t,x)=\sqrt{\frac{x(1-x)}{\rho^2(t)}}$ are the coefficients in \eqref{eq:sde_vary}, and $\widetilde{W}$ is a Brownian motion.  Sufficient conditions on the function $\rho$ for which the process  $(p_t)_{t\in[0,T]}$ satisfies a stochastic differential equation may be deduced from an integration by parts argument as in \cite{millet1989integration}.
\end{remark}

Next, we look at asymptotics.
The following analogue of Lemma \ref{L:Convphi} holds for any initial distribution $\mu_0$ of $X(0)$ and any demographic function $\rho$ that is bounded and positive. Note that in Lemma \ref{L:Convphi}, $\mu_0=\phi_{\alpha}$ depends on $\alpha$, but here $\mu_0$ is fixed.
\begin{prop}\label{prop:Convphi_vary}
Let $T>0$ and $\mu_0\in \mathcal{P}([0,1])$ be fixed, and the demographic function $\rho$ be as in Lemma \ref{L:MoranVary}.  The following  convergences in $\mathcal{P}([0,1])$ hold.
\begin{itemize}
\item[(i)] Suppose  $n_2$ is fixed and $\alpha\to\infty$. Then  $\mathcal{L}^{{\bf n}}  \to \delta_1$.
\item[(ii)] Suppose  $\alpha$ is fixed and $n_2\to\infty$. Then  $\mathcal{L}^{{\bf n}}  \to \delta_0$.
\item[(iii)] Suppose $\widetilde{\alpha}, c\in\R$ are fixed  and $\alpha=\widetilde{\alpha}n_2+c\to\infty$. Suppose $\P_{\mu_0}(X_T\in dy)$ has a density $p(T, \mu_0,y)\,dy$ and there exists a large deviation rate function $\mathcal{I}:[0,1]\to [0,\infty]$ such that for each $y\in[0,1]$,
\begin{linenomath*}
\begin{equation*}
\frac{\log p(T, \mu_0,y) }{n_2} \to \mathcal{I}(y) \qquad\text{as }n\to\infty.
\end{equation*}
\end{linenomath*}
Suppose also that 
$(1-y) \,e^{\mathcal{I}(y)}$ has a unique maximum at $y_*\in[0,1]$. 
Then $\mathcal{L}^{{\bf n}}  \to \delta_{y_*}$, and $\mathcal{I}(y_*)=\frac{1}{1-y_*}$.
\end{itemize}
\end{prop}

\begin{remark}\rm 
The assumptions in (iii) hold when $\rho(t) \equiv 1$ and $\mu_0=\phi_{\alpha}$ in \eqref{eq:phix}, i.e. constant population size with stationary initial condition. In this case,
$\mathcal{I}(y)=\widetilde{\alpha}\,y$ is the linear function, and
$\mathcal{L}^{{\bf n}}=\phi^{(n_1,n_2)}_{\alpha}$ in \eqref{eq:condphi2}.
The rate function has a phase transition at $\widetilde{\alpha}=1$, as shown in Lemma \ref{L:Convphi}. Namely, $y_*=0$ when $\widetilde{\alpha}\in (-\infty,1)$ and $y_*=1-1/\widetilde{\alpha}$ when $\widetilde{\alpha}\in (1,\infty)$. 
\end{remark}

\begin{remark}\rm
The large deviation principle for $\P_{\mu_0}(X_T\in dy)$ as $n_2\to\infty$ can be checked using the G\"artner-Ellis theorem \cite[Theorem 2.3.6]{demboLDP}. When it holds, the  rate function $\mathcal{I}$ is equal to the Legendre transform of the  function
\begin{linenomath*}
\begin{equation*}
\Lambda(\gamma)\coleq\lim_{n_2\to\infty}\frac{1}{n_2}\log \E_{\mu_0}[e^{\gamma\,n_2\,X_T}].
\end{equation*}
\end{linenomath*}
\end{remark}

\begin{proof}
A proof follows from that of Lemma \ref{L:Convphi}. 
Let $f\in C_b([0,1])$, a bounded continuous function on $[0,1]$. Then
\begin{linenomath*}
\begin{align}
\E_{{\bf n}}[f(p_0)]=\int_0^1 f(y)\,\mathcal{L}^{{\bf n}} (dy)
=&\,\frac{\E_{\mu_0}[f(X_T)\,X_T^{n_1}(1-X_T)^{n_2}]}{\E_{\mu_0}[X_T^{n_1}(1-X_T)^{n_2}]}.\label{Posterior_vary}
\end{align}
\end{linenomath*}

For part (i), note that if $(n_1,n_2)$ is fixed and $\alpha\to\infty$, then 
$\P_{\mu_0}(X_T<1-\epsilon)\to 0$ for any $\epsilon>0$ as in the proof of Theorem \ref{T:tildealpha>1}(i). Hence $\E_{\mu_0}[|f(X_T)-f(1)|] \to 0$ for any $f\in C_b([0,1])$.
In particular, $\P_{\mu_0}(X_T\in dy)\to\delta_1$. Hence $\mathcal{L}^{{\bf n}}$ tends to  $\delta_1$ in $\mathcal{P}([0,1])$, as $\alpha\to \infty$.

For part (ii), note that $(1-y)^{n_2}$ has maximum at $y=0$, and $y^{n_1}\,\P_{\mu_0}(X_T\in dy)$ does not depend on $n_2$. Hence $\mathcal{L}^{{\bf n}}$ tends to $\delta_0$ in $\mathcal{P}([0,1])$, as $n_2\to \infty$.

For part  (iii), the numerator of \eqref{Posterior_vary} is
\begin{linenomath*}
\begin{align*}
\E_{\mu_0}[f(X_T)\,X_T^{n_1}(1-X_T)^{n_2}] 
=&\, \int_0^1 f(y)\,y^{n_1} (1-y)^{n_2}\, \P_{\mu_0}(X_T\in dy) \\
\approx &\, \int_0^1 f(y)\,y^{n_1} \,[(1-y)\, e^{\mathcal{I}(y)}]^{n_2}\,\varphi(y)\,dy,
\end{align*}
\end{linenomath*}
for some function $\varphi$ such that $\frac{\ln \varphi(y) }{n_2}\to 0$, by assumptions of part (iii). Since
$(1-y) \,e^{\mathcal{I}(y)}$ has a unique maximum at $y_*$, 
$\lim_{n\to\infty}\E_{{\bf n}}[f(p_0)]=f(y_*)$ by \eqref{Posterior_vary} and a standard argument as in the proof of Lemma \ref{L:Convphi}. 
\end{proof}

By Proposition \ref{prop:BothType_vary} and  Proposition \ref{prop:Convphi_vary}, similar limiting results for the conditional coalescent process for scenarios (i) and (ii) hold for any positive function $\rho:\,[0,T]\to(0,\infty)$ that is \textit{continuous near the current time $T$}.  Note that $\rho$ is bounded away from zero on any compact time interval $[0,T]$ and therefore analogous approximations for the frequency process $(p_t)_{t\in[0,T]}$ under $\P_{\bf n}$ 
still hold, where the new approximating functions now involve the $h$ function in Remark \ref{Rk:SDE for p}.

More precisely, in scenario (i), Theorem \ref{T:scenario1} still holds, and Theorem \ref{T:Age1} still holds but with a possibly different limiting random variable. Hence $s_{n_2} \sim O(1/\alpha)$ is very small and $p_{\tau_i} \sim 1$, so $K_1\to 1$ by \eqref{xi_K1_vary}, and the single latent mutation for type 1 is very old.  

In scenario (ii), $\tau_{n_1}\sim O(1/n)$ is very small, and recalling that $\rho(T) = 1$, we have $K_1\approx 1+\sum_{k=2}^{n_1}\xi_k$ by \eqref{xi_K1_vary}, where $\{\xi_k\}$ are independent Bernoulli variables taking values in $\{0,1\}$ and having means $\frac{\theta_1}{\theta_1+k-1}$.  Theorem \ref{T:limitZ} with $\alpha \in \R$ fixed still holds, but the statement needs to be modified because the approximating process  $Z$ in \eqref{limitZ} will be replaced by another one that involves the $h$ function in Remark \ref{Rk:SDE for p}. 

Scenario (iii) is harder to analyze and we leave it for future work. We conjecture that if $y_*=0$, then the conditional genealogy behaves like scenario (ii); and if $y_*\in (0,1]$, then the conditional genealogy behaves like scenario (i).

\section{Conditional ancestral selection graph} \label{sec:condasg}

Our aim in this section is to see how the results of the previous section can be obtained from a different model: the ancestral selection graph.  We concentrate on the ancestry of focal allele $A_1$ and on constant population size, and proceed more heuristically than in the previous section.  

The ancestral selection graph is an augmented coalescent model for the joint distribution of the gene genealogy and the allelic states of the sample \citep{KroneAndNeuhauser1997,NeuhauserAndKrone1997}.  It includes the usual coalescent rate $1$ per pair of lineages and mutation rate $\theta/2$ per lineage.  Additionally, under the stationary model of Section~\ref{sec:intro}, it includes a \textit{branching} rate of $\lvert\alpha\rvert/2$ per lineage.  When a branching event occurs, the lineage splits into an \textit{incoming} lineage and a \textit{continuing} lineage.  One of these is \textit{real}, meaning it is included in the gene genealogy.  The other is \textit{virtual}, meaning it is there only to model the gene genealogy correctly with selection.  Which is which could be resolved if their allelic states were known: the incoming lineage is real if its allelic type is the one favored by selection, otherwise the continuing lineage is real.  But the allelic states are not known in the construction of the ancestral selection graph. 

The conditional ancestral selection graph models gene genealogies given a sample with allelic states specified \citep{Slade2000a,Slade2000b,Fearnhead2001,Fearnhead2002,StephensAndDonnelly2003,BaakeAndBialowons2008}.  In this case it is known which lineages are real and which are virtual.  This allows a simplification in which there is a reduced rate of branching and only virtual lineages of the disfavored type are produced \citep{Slade2000a}.  A second simplification is possible if mutation is parent-independent: then any lineage which mutates may be discarded \citep{Fearnhead2002}.  

We assume parent-independent mutation, specifically $\theta_1 = \theta \pi_1$ and $\theta_2 = \theta \pi_2$, with $\pi_1+\pi_2=1$.  Any two-allele mutation model can be restated in this way, leaving the stationary probability density \eqref{eq:phix} and the sampling probability \eqref{eq:qn1} unchanged.  But doing so introduces ``spurious mutations to one's own type'' \citep{Donnelly1986} or ``empty mutations'' \citep{BaakeAndBialowons2008} which occur only in the model and do not correspond to a biological process.  These are not latent mutations.  Including them allows us to discard real $A_2$ lineages and any virtual lineage once these mutate, but we must distinguish between empty and actual  mutations in the ancestry of $A_1$.  

The resulting conditional process tracks the numbers of real and virtual ancestral lineages from the present time $t=0$ back into the past.  Let $r_1(t)$, $r_2(t)$ and $v_i(t)$, where $i=1$ if $\alpha<0$ or $i=2$ if $\alpha>0$, be the numbers of real type-$1$, real type-$2$ and virtual type-$i$ lineages at past time $t$.  The process begins in state $r_1(0)=n_1$, $r_2(0)=n_2$, $v_i(0)=0$ and stops when $r_1(t)+r_2(t)=1$.  We suppress $t$ in what follows, and focus on the instantaneous transition rates of the process.   

The conditional ancestral process is obtained by considering rates of events in the unconditional process, which has total rate $(r_1+r_2+v_i)(\theta+\lvert\alpha\rvert+r_1+r_2+v_i-1)/2$, then weighting rates of events depending on how likely they are to produce the sample.  Rates of some events are down-weighted to zero.  For instance, the sample could not have been obtained if there were a coalescent event between lineages with different allelic types, whereas in the unconditional process these happen with rate $r_1r_2$ plus either $r_2v_1$ or $r_1v_2$, depending on whether $\alpha<0$ or $\alpha>0$.  

Rates of events for which the sample has a non-zero chance of being observed are up-weighted or down-weighted by ratios of sampling probabilities like 
\eqref{eq:qn1}.  This method of conditioning a Markov process on its eventual outcome is stated simply in \citet[p.\ 64]{KemenyAndSnell1960}, a familiar example being the Wright-Fisher diffusion conditioned on eventual fixation \citep[p.\ 89]{Ewens2004}, and is characterized more generally by Doob's h-transform \citep{Doob1957,Doob2001}.  In the conditional ancestral selection graph, the Markov process is the (unconditional) ancestral process of \citet{KroneAndNeuhauser1997} and the eventual outcome is the sample with allelic states specified.   

In our formulation, the samples and their ancestral lineages all are distinguishable, which we denote with a subscript ``o'' for ordered as in \citet{WakeleyEtAl2023}.  The probability of any particular allelic configuration in the ancestry of the sample, in which there are $r_1$ lineages of type $1$, $r_2$ lineages of type $2$ and $v_i$ lineages of type $i\in\{1,2\}$, is 
\begin{linenomath*}
\begin{align}
q_o(r_1,r_2,v_1) &= \int_0^1 x^{r_1+v_1} (1-x)^{r_2} \phi_\alpha(x) dx \quad \text{if } \,\alpha<0 \label{eq:qr1r2v1} \\[4pt]
q_o(r_1,r_2,v_2) &= \int_0^1 x^{r_1} (1-x)^{r_2+v_2} \phi_\alpha(x) dx \quad \text{if } \,\alpha>0 \label{eq:qr1r2v2}  
\end{align} 
\end{linenomath*}
with $\phi_\alpha(x)$ as in \eqref{eq:phix}.  Note, the additional binomial coefficient in the sampling probability \eqref{eq:qn1} is the number of possible orderings of a sample containing $n_1$ and $n_2$ copies of $A_1$ and $A_2$.  

The rate of any particular event with non-zero probability in the conditional process is the product of its rate in the unconditional process and a ratio of sampling probabilities from either \eqref{eq:qr1r2v1} or \eqref{eq:qr1r2v2}.  For event $(r_1,r_2,v_i)\to(r_1^\prime,r_2^\prime,v_i^\prime)$, the required ratio is $q_o(r_1^\prime,r_2^\prime,v_i^\prime)/q_o(r_1,r_2,v_i)$.  The denominator $q_o(r_1,r_2,v_i)$ is the probability of the sample given all events so far in the conditional ancestral process, which have led to the current state $(r_1,r_2,v_i)$, and the numerator $q_o(r_1^\prime,r_2^\prime,v_i^\prime)$ is the probability of the sample given these events and the event $(r_1,r_2,v_i)\to(r_1^\prime,r_2^\prime,v_i^\prime)$.  \ref{sec:asgdetails} provides the details of how the minimal ancestral process we use here to model latent mutations in the ancestry of the sampled copies of allele $A_1$ is obtained from the full conditional ancestral process, using the simplifications of \citet{Slade2000a} and \citet{Fearnhead2002}.

The resulting conditional ancestral process differs depending on whether $\alpha<0$ or $\alpha>0$ but in either case it includes five possible transitions from state $(r_1,r_2,v_i)$.  If $\alpha<0$,   
\begin{linenomath*}
\begin{equation}
(r_1,r_2,v_1) \to 
\begin{cases}
(r_1-1,r_2+1,v_1) & \text{at rate} \quad r_1 \frac{\theta\pi_1}{2} \frac{q_o(r_1-1,r_2+1,v_1)}{q_o(r_1,r_2,v_1)} \\[6pt]
(r_1-1,r_2,v_1) & \text{at rate} \quad \binom{r_1}{2} \frac{q_o(r_1-1,r_2,v_1)}{q_o(r_1,r_2,v_1)} \\[6pt]
(r_1,r_2,v_1+1) & \text{at rate} \quad (r_1+r_2+v_1) \frac{\lvert\alpha\rvert}{2} \frac{q_o(r_1,r_2,v_1+1)}{q_o(r_1,r_2,v_1)} \\[6pt]
(r_1,r_2,v_1-1) & \text{at rate} \quad {\left( v_1 \frac{\theta\pi_1}{2} + r_1 v_1 + \binom{v_1}{2} \right)} \frac{q_o(r_1,r_2,v_1-1)}{q_o(r_1,r_2,v_1)} \quad \\[6pt]
(r_1,r_2-1,v_1) & \text{at rate} \quad {\left( r_2 \frac{\theta\pi_2}{2} + \binom{r_2}{2} \right)} \frac{q_o(r_1,r_2-1,v_1)}{q_o(r_1,r_2,v_1)} \label{eq:asgn1alphanegcases} 
\end{cases}
\end{equation}
\end{linenomath*}
whereas if $\alpha>0$,  
\begin{linenomath*}
\begin{equation}
(r_1,r_2,v_2) \to 
\begin{cases}
(r_1-1,r_2+1,v_2) & \text{at rate} \quad \frac{r_1 \theta\pi_1}{2} \frac{q_o(r_1-1,r_2+1,v_2)}{q_o(r_1,r_2,v_2)} \\[6pt]
(r_1-1,r_2,v_2) & \text{at rate} \quad \binom{r_1}{2} \frac{q_o(r_1-1,r_2,v_2)}{q_o(r_1,r_2,v_2)} \\[6pt]
(r_1,r_2,v_2+1) & \text{at rate} \quad (r_1+r_2+v_2) \frac{\alpha}{2} \frac{q_o(r_1,r_2,v_2+1)}{q_o(r_1,r_2,v_2)} \\[6pt]
(r_1,r_2,v_2-1) & \text{at rate} \quad {\left( v_2 \frac{\theta\pi_2}{2} + r_2 v_2 + \binom{v_2}{2} \right)} \frac{q_o(r_1,r_2,v_2-1)}{q_o(r_1,r_2,v_2)} \quad \\[6pt]
(r_1,r_2-1,v_2) & \text{at rate} \quad {\left( r_2 \frac{\theta\pi_2}{2} + \binom{r_2}{2} \right)} \frac{q_o(r_1,r_2-1,v_2)}{q_o(r_1,r_2,v_2)} \label{eq:asgn1alphaposcases} 
\end{cases}
\end{equation}
\end{linenomath*}
which differ owing to the different resolutions of branching events when $\alpha<0$ versus $\alpha>0$.  We may note that the total rates of events in \eqref{eq:asgn1alphanegcases} and \eqref{eq:asgn1alphaposcases} are less than in the unconditional ancestral process because the conditional process has a reduced rate of branching \citep{Slade2000a} and because empty mutations do not change the number or types of ancestral lineages.  If $\alpha<0$, the total rate is $r_1 \theta\pi_1/2 + r_2 \lvert\alpha\rvert/2$ less, whereas if $\alpha>0$, it is $r_1 \theta\pi_1/2 + r_1 \alpha/2$ less.  

Asymptotic approximations for the ratios $q_o(r_1^\prime,r_2^\prime,v_i^\prime)/q_o(r_1,r_2,v_i)$ in these rates of events can be obtained using the results in \ref{sec:asymptotics}.  In the following three subsections we present approximations to the conditional ancestral process for our three scenarios of interest: (i) $\lvert\alpha\rvert$ large with $n_2$ fixed, (ii) $n_2$ large with $\alpha$ fixed, and (iii) both $\lvert\alpha\rvert$ and $n_2$ large with $\widetilde{\alpha} = \alpha/n_2$ fixed.  Because initially $r_2=n_2$, we consider $r_2$ large in the scenarios with $n_2$ large.  For each scenario, we compute the transition rates up to leading order in $\lvert\alpha\rvert$ or $r_2$, then consider how these conform to the corresponding results of Section~\ref{sec:bes}. 

\subsection{Scenario (i): strong selection, arbitrary sample size} \label{sec:condasgsub1}

Here $\lvert\alpha\rvert$ is large with $n_2$ fixed, along with $n_1$ and $\theta$.  In Section~\ref{sec:bessub1}, Theorem~\ref{T:scenario1}, we treated the ancestries of $A_1$ and $A_2$ simultaneously as $\alpha\to+\infty$, so that $A_1$ was favored and $A_2$ was disfavored.  Here we cover these same two possibilities by modeling the ancestry of $A_1$, using \eqref{eq:asgn1alphanegcases} when $A_1$ is disfavored ($\alpha<0$) and  \eqref{eq:asgn1alphaposcases} when $A_1$ is favored ($\alpha>0$).  We disregard the ancestry of the non-focal allele $A_2$ except insofar as it is needed to model events in the ancestry of $A_1$.     

When $\alpha<0$, using \eqref{eq:1F1s1} in \eqref{eq:asgn1alphanegcases} gives    
\begin{linenomath*}
\begin{equation}
(r_1,r_2,v_1) \to 
\begin{cases}
(r_1-1,r_2+1,v_1) & \text{at rate} \quad r_1 \frac{\lvert\alpha\rvert}{2} \frac{\theta\pi_1}{\theta\pi_1+r_1+v_1-1} \,+\, O{\left(1\right)} \\[6pt]
(r_1-1,r_2,v_1) & \text{at rate} \quad r_1 \frac{\lvert\alpha\rvert}{2} \frac{r_1-1}{\theta\pi_1+r_1+v_1-1} \,+\, O{\left(1\right)} \\[6pt]
(r_1,r_2,v_1+1) & \text{at rate} \quad \frac{(r_1+r_2+v_1)(\theta\pi_1+r_1+v_1)}{2} \,+\, O{\left(\lvert\alpha\rvert^{-1}\right)} \\[6pt]
(r_1,r_2,v_1-1) & \text{at rate} \quad v_1 \frac{\lvert\alpha\rvert}{2} \frac{\theta\pi_1+2r_1+v_1-1}{\theta\pi_1+r_1+v_1-1} \,+\, O{\left(1\right)} \\[6pt]
(r_1,r_2-1,v_1) & \text{at rate} \quad r_2 \frac{\theta\pi_2}{2} + \binom{r_2}{2} \,+\, O{\left(\lvert\alpha\rvert^{-1}\right)} \label{eq:asgn1largealphanegcases} 
\end{cases}
\end{equation}
\end{linenomath*}
for the ancestry of a disfavored allele under strong selection (as $\alpha\to-\infty$).  Latent mutations and coalescent events occur with rates proportional to $\lvert\alpha\rvert$.  Virtual lineages are removed similarly quickly but are produced at a much lower rate.  So $v_1$ will stay zero during the $O(1/\lvert\alpha\rvert)$ time it takes for the requisite $n_1$ latent mutations or coalescent events to occur.  Then the analogous result to \eqref{E:one-step type2}, namely \eqref{eq:pcoalj} and \eqref{eq:K1sum}, follows from the first two lines of \eqref{eq:asgn1largealphanegcases}.  Coalescence and mutation among the copies of $A_2$ occur at the slower rate, so none of these should occur before all the type-$1$ lineages disappear.  These results were first suggested in \citet{Wakeley2008}.

When $\alpha>0$, using \eqref{eq:1F1s2} in \eqref{eq:asgn1alphaposcases} gives    
\begin{linenomath*}
\begin{equation}
(r_1,r_2,v_2) \to 
\begin{cases}
(r_1-1,r_2+1,v_2) & \text{at rate} \quad r_1 \frac{1}{\alpha} \frac{\theta\pi_1 (\theta\pi_2+r_2+v_2)}{2} \,+\, O{\left(\alpha^{-2}\right)} \\[6pt](r_1-1,r_2,v_2) & \text{at rate} \quad \binom{r_1}{2} \,+\, O{\left(\alpha^{-1}\right)} \\[6pt]
(r_1,r_2,v_2+1) & \text{at rate} \quad \frac{(r_1+r_2+v_2)(\theta\pi_2+r_2+v_2)}{2} \,+\, O{\left(\alpha^{-1}\right)} \\[6pt]
(r_1,r_2,v_2-1) & \text{at rate} \quad v_2 \frac{\alpha}{2} \frac{\theta\pi_2+2r_2+v_2-1}{\theta\pi_2+r_2+v_2-1} \,+\, O{\left(1\right)} \\[6pt]
(r_1,r_2-1,v_2) & \text{at rate} \quad r_2 \frac{\alpha}{2} \frac{\theta\pi_2+r_2-1}{\theta\pi_2+r_2+v_2-1} \,+\, O{\left(1\right)} 
\label{eq:asgn1largealphaposcases} 
\end{cases}
\end{equation}
\end{linenomath*}
for the ancestry of a favored allele under strong selection (as $\alpha\to+\infty$).  Now $A_2$ is undergoing the fast process just described for $A_1$ in \eqref{eq:asgn1largealphanegcases}, so these lineages will disappear quickly.  Again the rate of removal of virtual lineages greatly exceeds their rate of production.  In $O(1/\alpha)$ time, the ancestral state will become $(r_1,r_2,v_2)=(n_1,0,0)$.  But now with $A_1$ favored, the rates of coalescence and latent mutation differ by a factor of $\alpha$, so the first $n_1-1$ events will be coalescent events, followed by a long wait for a single latent mutation with rate $\theta^2\pi_1\pi_2/(2\alpha)$ as in Theorem~\ref{T:Age1}.

\subsection{Scenario (ii): arbitrary selection, large sample size} \label{sec:condasgsub2}

Here $n_2$ is large with $\alpha$ fixed, along with $n_1$ and $\theta$.  Because $r_2=n_2$ at the start of the ancestral process, we present rates of events to leading order in $1/r_2$.  In Section~\ref{sec:bessub2} we deferred this scenario to Section~\ref{sec:bessub3}, because in the limit it is equivalent to $\widetilde{\alpha}=0$.  Of course, there are two ways for $\widetilde{\alpha}$ to approach zero, and the sign of $\widetilde{\alpha}$ matters in \eqref{limitZ} for any $\widetilde{\alpha}$ not strictly equal to zero.  Here we consider the two cases, $\alpha<0$ and $\alpha>0$, separately.

When $\alpha<0$, using \eqref{eq:1F1approxw} in \eqref{eq:asgn1alphanegcases} gives    
\begin{linenomath*}
\begin{equation}
(r_1,r_2,v_1) \to 
\begin{cases}
(r_1-1,r_2+1,v_1) & \text{at rate} \quad r_1 \frac{r_2}{2} \frac{\theta\pi_1}{\theta\pi_1+r_1+v_1-1} \,+\, O{\left(1\right)} \\[6pt]
(r_1-1,r_2,v_1) & \text{at rate} \quad r_1 \frac{r_2}{2} \frac{r_1-1}{\theta\pi_1+r_1+v_1-1} \,+\, O{\left(1\right)} \\[6pt]
(r_1,r_2,v_1+1) & \text{at rate} \quad \frac{\lvert\alpha\rvert(\theta\pi_1+r_1+v_1)}{2} \,+\, O{\left(r_2^{-1}\right)} \\[6pt]
(r_1,r_2,v_1-1) & \text{at rate} \quad v_1 \frac{r_2}{2} \frac{\theta\pi_1+2r_1+v_1-1}{\theta\pi_1+r_1+v_1-1} \,+\, O{\left(1\right)} \\[6pt]
(r_1,r_2-1,v_1) & \text{at rate} \quad \frac{r_2^2}{2} \,+\, O{\left(r_2\right)} 
\label{eq:asgn1largen2alphanegcases} 
\end{cases}
\end{equation}
\end{linenomath*}
This differs from the neutral case \citep{WakeleyEtAl2023} only by the possibility of virtual lineages. As in \eqref{eq:asgn1largealphanegcases}, these will be removed quickly if they are produced. The process of latent mutation and coalescence happens in $O(1/r_2)$ time, with relative rates in the first two lines of \eqref{eq:asgn1largen2alphanegcases} again giving \eqref{eq:pcoalj} and \eqref{eq:K1sum}.  Because $r_2 \to \infty$, this approximation will hold long enough for the required fixed number of events among the $A_1$ lineages to occur, despite the rapid decrease of $r_2$ in the last line of \eqref{eq:asgn1largen2alphanegcases}.  A proof of this is given in \citet[Appendix]{WakeleyEtAl2023}.  Theorem~\ref{T:JointConvergence} addresses the corresponding issues for the model of Section~\ref{sec:bes}.

When $\alpha>0$, using \eqref{eq:1F1approxw} in \eqref{eq:asgn1alphaposcases} gives    
\begin{linenomath*}
\begin{equation}
(r_1,r_2,v_2) \to 
\begin{cases}
(r_1-1,r_2+1,v_2) & \text{at rate} \quad r_1 \frac{r_2}{2} \frac{\theta\pi_1}{\theta\pi_1+r_1-1} \,+\, O{\left(1\right)} \\[6pt]
(r_1-1,r_2,v_2) & \text{at rate} \quad r_1 \frac{r_2}{2} \frac{r_1-1}{\theta\pi_1+r_1-1} \,+\, O{\left(1\right)} \\[6pt]
(r_1,r_2,v_2+1) & \text{at rate} \quad \frac{r_2 \lvert\alpha\rvert}{2} \,+\, O{\left(1\right)} \\[6pt]
(r_1,r_2,v_2-1) & \text{at rate} \quad v_2 r_2 \,+\, O{\left(1\right)} \\[6pt]
(r_1,r_2-1,v_2) & \text{at rate} \quad \frac{r_2^2}{2} \,+\, O{\left(r_2\right)}  
\label{eq:asgn1largen2alphaposcases} 
\end{cases}
\end{equation}
\end{linenomath*}
which differs from \eqref{eq:asgn1largen2alphanegcases} in two ways.  Now the rate of production of virtual lines is non-negligible.  But here their presence does not affect the rates of latent mutation and coalescence.  Again we have \eqref{eq:pcoalj} and \eqref{eq:K1sum}, and the process of latent mutation and coalescence happens in $O(1/r_2)$ time.

\subsection{Scenario (iii): strong selection, large sample size} \label{sec:condasgsub3}

Here both $\lvert\alpha\rvert$ and $n_2$ are large with $\widetilde{\alpha} = \alpha/n_2$ fixed, along with $n_1$ and $\theta$.  Again since the process begins with $r_2=n_2$, we present rates of events to leading order in $1/r_2$.  Because the conditional ancestral process differs for $\alpha<0$ versus $\alpha>0$, i.e.\ with \eqref{eq:asgn1alphanegcases} and \eqref{eq:asgn1alphaposcases}, and the asymptotic approximation we use for the hypergeometric function differs for $\widetilde{\alpha}<1$ versus $\widetilde{\alpha}>1$, i.e.\ with \eqref{eq:1F1ws1} and \eqref{eq:1F1ws3}, here we have three cases.  Note these are the same three cases in \eqref{eq:qn1n2alphatilde1}, \eqref{eq:qn1n2alphatilde2} and \eqref{eq:qn1n2alphatilde3}. 

When $\widetilde{\alpha}<0$, using \eqref{eq:1F1ws1} in \eqref{eq:asgn1alphanegcases} gives    
\begin{linenomath*}
\begin{equation}
(r_1,r_2,v_1) \to 
\begin{cases}
(r_1-1,r_2+1,v_1) & \text{at rate} \quad r_1 \frac{r_2{\left(1+\lvert\widetilde{\alpha}\rvert\right)}}{2} \frac{\theta\pi_1}{\theta\pi_1+r_1+v_1-1} \,+\, O{\left(1\right)} \\[6pt]
(r_1-1,r_2,v_1) & \text{at rate} \quad r_1 \frac{r_2{\left(1+\lvert\widetilde{\alpha}\rvert\right)}}{2} \frac{r_1-1}{\theta\pi_1+r_1+v_1-1} \,+\, O{\left(1\right)} \\[6pt]
(r_1,r_2,v_1+1) & \text{at rate} \quad \frac{\lvert\widetilde{\alpha}\rvert(\theta\pi_1+r_1+v_1)}{2{\left(1+\lvert\widetilde{\alpha}\rvert\right)}} \,+\, O{\left(r_2^{-1}\right)} \\[6pt]
(r_1,r_2,v_1-1) & \text{at rate} \quad v_1 \frac{r_2{\left(1+\lvert\widetilde{\alpha}\rvert\right)}}{2} \frac{\theta\pi_1+2r_1+v_1-1}{\theta\pi_1+r_1+v_1-1} \,+\, O{\left(1\right)} \\[6pt]
(r_1,r_2-1,v_1) & \text{at rate} \quad \frac{r_2^2}{2} \,+\, O{\left(r_2\right)} 
\label{eq:asgn1alphatildecase1} 
\end{cases}
\end{equation}
\end{linenomath*}
which is comparable to \eqref{eq:asgn1largealphanegcases} and \eqref{eq:asgn1largen2alphanegcases}.  Again we may effectively ignore virtual lineages.  The rates of latent mutation and coalescence in \eqref{eq:asgn1largealphanegcases} and \eqref{eq:asgn1largen2alphanegcases} differ only by the interchange of $r_2$ for $\lvert\alpha\rvert$.   In \eqref{eq:asgn1alphatildecase1}, the factor $r_2{\left(1+\lvert\widetilde{\alpha}\rvert\right)}$ encompasses the effects of both.  The larger $\lvert\widetilde{\alpha}\rvert$ is, the more quickly these events will occur, and again \eqref{eq:pcoalj} and \eqref{eq:K1sum} describe the number of latent mutations.

When $0<\widetilde{\alpha}<1$, using \eqref{eq:1F1ws1} in \eqref{eq:asgn1alphaposcases} gives    
\begin{linenomath*}
\begin{equation}
(r_1,r_2,v_2) \to 
\begin{cases}
(r_1-1,r_2+1,v_2) & \text{at rate} \quad r_1 \frac{r_2{\left(1-\widetilde{\alpha}\right)}}{2} \frac{\theta\pi_1}{\theta\pi_1+r_1-1} \,+\, O{\left(1\right)} \\[6pt]
(r_1-1,r_2,v_2) & \text{at rate} \quad r_1 \frac{r_2{\left(1-\widetilde{\alpha}\right)}}{2} \frac{r_1-1}{\theta\pi_1+r_1-1} \,+\, O{\left(1\right)} \\[6pt]
(r_1,r_2,v_2+1) & \text{at rate} \quad \frac{r_2 \widetilde{\alpha}}{2} \,+\, O{\left(1\right)} \\[6pt]
(r_1,r_2,v_2-1) & \text{at rate} \quad v_2 r_2 \,+\, O{\left(1\right)} \\[6pt]
(r_1,r_2-1,v_2) & \text{at rate} \quad \frac{r_2^2}{2} \,+\, O{\left(r_2\right)}  
\label{eq:asgn1alphatildecase2} 
\end{cases}
\end{equation}
\end{linenomath*}
which is comparable to \eqref{eq:asgn1largen2alphaposcases}.  In contrast to \eqref{eq:asgn1alphatildecase1}, now with $A_1$ favored, the larger $\widetilde{\alpha}$ is (i.e.\ the closer it is to $1$) the smaller the rates of latent mutation and coalescence become.  Otherwise, for any given $\widetilde{\alpha}$, the same conclusions regarding latent mutations and their timing follow from \eqref{eq:asgn1alphatildecase2} as from \eqref{eq:asgn1alphatildecase1}, and these conform to what is stated in Theorem~\ref{T:JointConvergence}.

When $\widetilde{\alpha}>1$, using \eqref{eq:1F1ws3} in \eqref{eq:asgn1alphaposcases} gives    
\begin{linenomath*}
\begin{equation}
(r_1,r_2,v_2) \to 
\begin{cases}
(r_1-1,r_2+1,v_2) & \text{at rate} \quad r_1 \frac{\theta\pi_1}{2} \frac{1}{\widetilde{\alpha}-1} \,+\, O{\left(r_2^{-1}\right)} \\[6pt]
(r_1-1,r_2,v_2) & \text{at rate} \quad \binom{r_1}{2} \frac{\widetilde{\alpha}}{\widetilde{\alpha}-1} \,+\, O{\left(r_2^{-1}\right)} \\[6pt]
(r_1,r_2,v_2+1) & \text{at rate} \quad \frac{r_2}{2} \,+\, O{\left(1\right)} \\[6pt]
(r_1,r_2,v_2-1) & \text{at rate} \quad v_2 r_2  \widetilde{\alpha} \,+\, O{\left(1\right)} \\[6pt]
(r_1,r_2-1,v_2) & \text{at rate} \quad \frac{r_2^2}{2} \widetilde{\alpha} \,+\, O{\left(r_2\right)}  
\label{eq:asgn1alphatildecase3} 
\end{cases}
\end{equation}
\end{linenomath*}
which paints a very different picture.  Whereas \eqref{eq:asgn1largealphanegcases}, \eqref{eq:asgn1largen2alphanegcases}, \eqref{eq:asgn1largen2alphaposcases}, \eqref{eq:asgn1alphatildecase1} and \eqref{eq:asgn1alphatildecase2} all give the Ewens sampling result described by \eqref{eq:pcoalj} and \eqref{eq:K1sum} and have these events occurring quickly on the coalescent time scale, \eqref{eq:asgn1alphatildecase3} is rather like \eqref{eq:asgn1largealphaposcases} in that the rates of latent mutation and coalescence are too slow to register on the time scale of events involving the non-focal allele $A_2$.  The overwhelmingly most frequent events in \eqref{eq:asgn1alphatildecase3} will be coalescent events between $A_2$ lineages at rate $\propto r_2^2$, so an effectively instantaneous transition will occur from $r_2$ large to $r_2$ comparable to $r_1$.  Then this case \eqref{eq:asgn1alphatildecase3} will collapse quickly to the corresponding case \eqref{eq:asgn1largealphaposcases} where coalescence without mutation will happen among the $A_1$ followed by a long wait for a single latent mutation. For the model in Section~\ref{sec:(iii)b}, this is described by Theorem~\ref{T:tildealpha>1} and Theorem~\ref{T:JointConvergence2}.  Finally we may note that initially the rates of latent mutation and coalescence in \eqref{eq:asgn1alphatildecase3} are precisely those predicted for the model in Section ~\ref{sec:(iii)b} from \eqref{E:Evo_type1} starting at $p_0 \to 1- 1/\widetilde{\alpha}$ as specified for $\widetilde{\alpha}\in(1,\infty)$ in \eqref{Dirac_initial}.  

\section{Discussion}

In this paper, we have considered a two allele model at a single genetic locus subject to recurrent mutation and selection in a large haploid population with possibly time-varying size.  We assumed that a sample of size $n$ was drawn uniformly from an infinite population under the diffusion approximation.  By extending the framework of \citet{BartonEtAl2004}, we described the asymptotic behaviors of the conditional genealogy and the number of latent mutations of the sample, given the sample frequencies of the two alleles.  This moves beyond what is in \citet{WakeleyEtAl2023} by the inclusion of selection and by the use of an entirely different model, i.e.\ coalescence in a random background \citep{BartonEtAl2004}.  This yields novel results. For example, in the strong selection case in which the selection strength $\alpha$ is proportional to the sample size $n$ and both go to infinity (our scenario (iii)), the genealogy of the rare allele can be described in terms of a Cox-Ingersoll-Ross (CIR) diffusion with an initial Gamma distribution. 

The concept of rare alleles in this paper and in \citet{WakeleyEtAl2023} is the same as the one considered by \citet{JoyceAndTavare1995} and \citet{Joyce1995}.  It focuses on the counts of the alleles in a large sample rather than their relative frequencies in the population.  In scenarios (ii) and (iii) we consider a fixed number $n_1$ of the rare type 1 when the sample size $n$ tends to infinity.  \citet{JoyceAndTavare1995} considered rare alleles in a large sample drawn from the stationary distribution of a $d$-dimensional Wright-Fisher diffusion with selection  and mutation.  They showed that the counts of rare alleles, from different latent mutations in our terminology, have approximately independent Poisson distributions with parameters that do \textit{not} depend on the selection parameters, and that the Ewens sampling formula describes their distribution.  Their model with $d=2$ and genic selection corresponds to our scenario (ii). Our results for very strong selection ($\alpha\to\infty$) in scenario (iii) differ from those of \citet{JoyceAndTavare1995} in that the rare-allele sampling probabilities \eqref{eq:qn1n2alphatilde1}, \eqref{eq:qn1n2alphatilde2} and \eqref{eq:qn1n2alphatilde3} do depend on selection.  Interestingly, the number of latent mutations given $n_1$ still follows the Ewens sampling formula when $\lim_{n\to\infty}\alpha/n \in (-\infty,1)$.  But this is not true when $\lim_{n\to\infty}\alpha/n\in (1,\infty)$, in which case the number of latent mutations is always $k_1 \equiv 1$.

Some of our results for rare alleles have empirical relevance, specifically those for scenario (ii) including their robustness to time-varying population size demonstrated in Section~\ref{S:varying}, and those for scenario (iii) with $\widetilde{\alpha}<0$.  In scenario (ii), as $n$ increases for fixed but arbitrary $\alpha$, the distributions of latent mutations and the ages of those latent mutations become identical to those for neutral alleles described in \citet{WakeleyEtAl2023}.  Our results also show that selection does have an effect in this case, but it is only to raise or lower the rare-allele sampling probability \eqref{eq:qn1n2largen2} by the constant factor $C$ for every value of $n_1$.  This relative insensitivity to selection suggests confidence in using rare alleles for demographic inference and genome-wide association studies \citep{OConnorEtAl2015,NaitSaadaEtAl2020,ZaidiAndMathieson2020}.  \citet{SlatkinAndRannala1997b}, who obtained the Ewens sampling formula result for rare deleterious alleles by assuming they evolve independently according to a linear birth-death process, cf.\ \citet{SlatkinAndRannala1997a}, suggested that deviations from this neutral prediction at two human-disease-associated loci were due to population growth.  \cite{ReichAndLander2001} made a similar argument for a number of other disease-associated loci starting from the mutation-selection balance model of \citet{HartlAndCampbell1982} and \citet{Sawyer1983} which also gives the Ewens sampling formula result for rare disease alleles.   

Our exploration of time-varying populations in Section~\ref{S:varying}, namely the robustness of the Ewens sampling formula result for the number of latent mutations, suggests that rare alleles may not always be well suited for demographic inference.  With only a mild constraint on the trajectory of population sizes through time, increasing the sample size will eventually make the distribution of latent mutations of rare alleles look as if the population size has been constant at its current size.  There is no doubt that demographic inferences improve as sample sizes increase.  What Section~\ref{S:varying} implies is that these improvements will not come from focusing exclusively on the lower end of sample allele frequencies (i.e.\ any fixed $n_1$ as $n\to\infty$).  How relevant this is for a given sample will depend on the actual ages of its latent mutations and the degree of population-size change between those times and the present.  To illustrate, consider the $O(1/n)$ ages of latent mutations under the exponential growth model with rate $\beta$.  If $\beta/n \ll 1$, the ancestral process of tracing back to these mutations will be complete before the population has changed much in size and the results of Section~\ref{S:varying} will hold.  But this is clearly not the case for the \textit{gnomAD} data in \citet{WakeleyEtAl2023} and \citet{SeplyarskiyEtAl2023}.  The distribution of $n_1$ in the non-Finnish European sample with $n=114K$ is well fit by $\beta/n = 3$.  See for example Fig.~3 in \citet{SeplyarskiyEtAl2023}.  Sample sizes would need to be orders of magnitude greater for the results in Section~\ref{S:varying} to hold in this case. 

Scenario (iii) with $\widetilde{\alpha}<0$ is applicable to strongly deleterious alleles.  An appreciable fraction of new mutations are strongly deleterious \citep{EyreWalkerAndKeightley2007,KimEtAl2017,WeghornEtAl2019,DuklerEtAl2022}.  Previous theoretical work includes \citet{Nei1968}, who found a gamma density analogous to ours in Lemma~\ref{L:Asymp_Initial} but for the population allele frequency of partially recessive lethal mutations, and \citet{CharlesworthAndHill2019}, who used Nei's approximation to derive the negative binomial distribution for $n_1$, our \eqref{eq:qn1n2alphatilde1}.  In this case, \eqref{eq:qn1n2alphatilde1ratio} shows that the sampling probabilities of rare alleles fall off quickly as $n_1$ grows: each additional copy of $A_1$ in the sample lowers its probability by a factor of $1/(1+\lvert\widetilde{\alpha}\rvert)$ compared to the neutral case.  Even so, the distribution of $k_1$ given $n_1$ follows the Ewens sampling formula.  \citet{HartlAndCampbell1982} and \citet{Sawyer1983} obtained similar results by assuming that both selection and mutation are strong.  Our analysis of scenario (iii) with $\widetilde{\alpha}<0$ also shows that latent mutations of rare strongly deleterious alleles are especially young: selection speeds up the ancestral process of latent mutation by a factor of $1+\lvert\widetilde{\alpha}\rvert$ on top of the factor of $n$ already present under neutrality.   This is most easily seen by comparing the first two lines of \eqref{eq:asgn1alphatildecase1} to the first two lines of \eqref{eq:asgn1largen2alphanegcases}.  

Our results for scenario (i) with $\alpha<0$, which hold as $\alpha\to-\infty$ for arbitrary sample size and alleles at any sample frequencies, are also applicable to strongly deleterious alleles.  They are similar to the results just discussed for scenario (iii) with $\widetilde{\alpha}<0$.  We expect that our results for very strong positive selection, i.e.\ scenario (i) with $\alpha>0$ and scenario (iii) with $\widetilde{\alpha}>0$, will be of limited applicability.  Mutations to strongly positively selected alleles are uncommon and observing such an allele a small number of times in a very large sample would be exceedingly unlikely. 

Many open questions remain.
\citet{Joyce1995} obtained a result similar to that of \citet{JoyceAndTavare1995}, for a Wright-Fisher diffusion with selection and infinite-alleles mutation.  This diffusion process is a particular case of the Fleming-Viot process (see \citet{ethier1993fleming} for a review) and it has a unique stationary distribution denoted $\nu_{\rm selec}$.  \citet{Joyce1995} considered a large sample of size $n$ drawn from $\nu_{\rm selec}$.  Let $C_b(n)\in\Z_+^b$ be the first $b$ allele counts in a sample of size $n$ drawn from the stationary distribution, and $K_n$ be the total number of alleles in the sample.  \citet{Joyce1995} showed that for any fixed $b$, the distribution of $(C_b(n),\,K_n)$ under $\nu_{\rm selec}$ is arbitrarily close to that under the neutral model.  It would be interesting to know if analogous results for our scenario (iii) also hold for the infinite allele model.  In particular, is there a threshold for the selection strength relative to $n$ that controls whether selection is washed out or not in the limit as $n\to\infty$?

For time-varying populations, 
little is known in scenario (iii).
For example, will the assumptions in 
Proposition \ref{prop:Convphi_vary} hold for a general demographic function? 
Will there be a phase transition for the value of $y_*$ in terms of $\widetilde{\alpha}$ and if so, what will determine the phase transition?
Also, both our results and those of \citet{JoyceAndTavare1995} and \citet{Joyce1995} are for the infinite-population diffusion limit.  Further consideration of the issues raised in Section~\ref{sec:relate} is needed to assess the relevance of these results to various kinds of finite populations. 

The critical case $\widetilde{\alpha}=1$ in scenario (iii) is omitted in this paper. Results for this case are expected to lie between those of $\widetilde{\alpha}>1$ and $\widetilde{\alpha}<1$, and require more in-depth asymptotic analysis. For example, one can first obtain asymptotic results for the hypergeometric function in \eqref{eq:1F1ws1}-\eqref{eq:1F1ws3} for the case $\widetilde{\alpha}=1$, and then follow the argument in Lemma \ref{L:Asymp_Initial} to obtain the asymptotic of the expectation $\E_{\bf n}[p_0]$ as $n_2\to\infty$ in this critical case.
Lemma \ref{L:Asymp_Initial} 
asserts that $\E_{\bf n}[p_0]=O(1)$ when $\widetilde{\alpha}>1$ and $\E_{\bf n}[p_0]=O(1/n_2)$ when $\widetilde{\alpha}<1$. We conjecture that $\E_{\bf n}[p_0]=O(n_2^{-\sigma})$ for some $\sigma\in (0,1)$ in the critical case. 

Finally, we have ignored the possibility of spatial structure.  Spatially heterogeneous populations in which reproduction rates, death rates, mutation rates and selection strength can depend both on spatial position and local population density present challenges.  
This is because the population dynamics now take place in high or infinite dimension \citep{hallatschek2008gene,Barton2010, MR3582808, louvet2023measurevalued, etheridge2023looking}.  For example, the spatial version of \eqref{eq:sde}, the stochastic Fisher-Kolmogorov-Petrovsky-Piscunov (FKPP) equation introduced by \citet{shiga1988stepping}, is a stochastic partial differential equation that arises as the scaling limit of various discrete models under weak selection \citep{mueller1995stochastic, MR3582808, MR4278798}.  Under the stochastic FKPP, \citet{hallatschek2008gene} and \citet{MR3582808} studied the backward-time lineage dynamics of a single sample individual, conditioned on knowing its type.  It would be interesting to see if our results in this paper can be extended to spatial stochastic models with selection.

\section*{Acknowledgements}
We thank Alison Etheridge for raising the question about the applicability of our limiting results to Wright-Fisher reproduction (cf.\ Section~\ref{sec:relate}).  We also thank Shamil Sunyaev, Evan Koch and Joshua Schraiber for helpful discussions, and Daniel Rickert and Kejia Geng for assistance in producing the figures.  Finally, we thank two anonymous reviewers for their insightful comments.  This research was partially supported by National Science Foundation grants DMS-1855417 and DMS-2152103, and Office of Naval Research grant N00014-20-1-2411 to Wai-Tong (Louis) Fan.



\appendix
%\section{Appendix} \label{sec:appendix}
%\setcounter{equation}{0}
%\renewcommand{\theequation}{A\arabic{equation}}

\section{Asymptotic approximations used in the text} \label{sec:asymptotics}

From the series expansion for a ratio of gamma functions with a common large parameter, 6.1.47 in \citet{AbramowitzAndStegun1964} or equation (1) in \citet{TricomiAndErdelyi1951}, we have     
\begin{linenomath*}
\begin{equation} 
\frac{\Gamma(a+n_2)}{\Gamma(b+n_2)} = n_2^{a-b} \left( 1 + \frac{(a-b)(a+b-1)}{2n_2} + O\left(n_2^{-2}\right) \right) \label{eq:gammaratio}
\end{equation}
\end{linenomath*}
for constants $a$ and $b$ which will depend on the application.  For example, we can apply \eqref{eq:gammaratio} twice in the sampling probability \eqref{eq:qn1} when $n_2$ is large: once with $a=n_1+1$ and $b=1$ (in the binomial coefficient) and once with $a=\theta_2$ and $b=\theta_1+\theta_2+n_1$.

The confluent hypergeometric function is commonly defined in terms of the integral 
\begin{linenomath*}
\begin{equation} 
_{1}F_{1}\left(a;b;z\right) = \frac{\Gamma(b)}{\Gamma(a)\Gamma(b-a)} \int_{0}^{1} e^{zu} u^{a-1} (1-u)^{b-a-1} du \label{eq:1F1integral}
\end{equation}
\end{linenomath*}
or in terms of the series
\begin{linenomath*}
\begin{equation} 
_{1}F_{1}\left(a;b;z\right) = \sum_{k=0}^{\infty} \frac{a^{(k)}z^k}{b^{(k)}k!} \label{eq:1F1series}
\end{equation}
\end{linenomath*}
which converges for all $z \in \mathbb{R}$ and $b>a>0$, where $a^{(k)}$ is the rising factorial $a(a+1)\cdots(a+k-1)$ with $a^{(0)}=1$.  Again $a$ and $b$ depend on the context, e.g.\ as in \eqref{eq:C} and \eqref{eq:qn1}.  The parameter $z$ corresponds to the selection parameter $\alpha$.

For large $\lvert \alpha \rvert$ and with constant $a$ and $b$, 
\begin{linenomath*}
\begin{subnumcases}{_{1}F_{1}\left(a;b;\alpha\right) = }
\frac{\Gamma(b)}{\Gamma(b-a)} {\lvert\alpha\rvert}^{-a} \left( 1 - \frac{a(b-a-1)}{\lvert\alpha\rvert} + O\left({\lvert\alpha\rvert}^{-2}\right) \right) & if $\,\alpha<0$ \label{eq:1F1s1} \\[6pt]
\frac{\Gamma(b)}{\Gamma(a)} e^\alpha \alpha^{a-b} \left( 1 - \frac{(b-a)(a-1)}{\alpha} + O\left(\alpha^{-2}\right) \right) & if $\,\alpha>0$ \label{eq:1F1s2}
\end{subnumcases} 
\end{linenomath*}
where the middle, neutral case is given only for completeness.  Equation \eqref{eq:1F1s1} is from (4.1.2) in \citet{Slater1960}, and \eqref{eq:1F1s2} is from (4.1.6) in \citet{Slater1960} or may be obtained from \eqref{eq:1F1s1} using Kummer's first theorem which appears as (1.4.1) in \citet{Slater1960}.

For large $n_2$, with constants $a$, $b$ and $z$, 
\begin{linenomath*}
\begin{equation} 
_{1}F_{1}\left(a;b+n_2;z\right) = 1 + \frac{a z}{n_2} + O\left(n_2^{-2}\right) \label{eq:1F1approxw}
\end{equation}
\end{linenomath*}
directly from \eqref{eq:1F1series}.

For large $n_2$ and $\alpha=\widetilde{\alpha}n_2$, with constants $a$ and $b$, 
\begin{linenomath*}
\begin{subnumcases}{_{1}F_{1}\left(a;b+n_2;\widetilde{\alpha}n_2\right) \approx }
\left(1-\widetilde{\alpha}\right)^{-a} & if $\,\widetilde{\alpha}<1$ \quad \label{eq:1F1ws1} \\[6pt]
\frac{\sqrt{2\pi}}{\Gamma(a)} \left(1-\frac{1}{\widetilde{\alpha}}\right)^{a-1} \left(\frac{1}{\widetilde{\alpha}}\right)^{b-a+n_2} n_2^{a-\frac{1}{2}} e^{n_2\left(\widetilde{\alpha}-1\right)}  & if $\,\widetilde{\alpha}>1$ \quad \label{eq:1F1ws3}
\end{subnumcases} 
\end{linenomath*}
which we present only to leading order for simplicity.  Equation \eqref{eq:1F1ws1} follows from \eqref{eq:1F1series} and \eqref{eq:1F1ws3} was obtained by applying Laplace's method to the integral in \eqref{eq:1F1integral} for this case.

\section{Proofs of Lemma \ref{L:Convphi}, Lemma \ref{L:Asymp_Initial} and Lemma 
 \ref{L:MoranVary}} \label{sec:lemmaproofs}

\begin{proof}[Proof of Lemma \ref{L:Convphi}]
Part (ii) then follows from part (iii) with $\widetilde{\alpha}=0$, and  the proof of part (i) follows from a similar argument.

To prove part (iii), we let $a=\theta_1+n_1$ for simplicity.
The function $\phi^{(n_1,n_2)}_{\alpha}$ is a constant multiple of  the function 
\begin{linenomath*}
\begin{align*}
 x^{a - 1} (1-x)^{n_2+\theta_2 - 1} e^{\alpha x} =&\,
\left[(1-x)e^{\widetilde{\alpha} x}\right]^{n_2}\,
x^{a-1}(1-x)^{\theta_2-1} e^{cx}\\
=&\,
e^{n_2\,S(x)}\,
x^{a-1}(1-x)^{\theta_2-1} e^{cx},
\end{align*}
\end{linenomath*}
where the function $S:\,[0,1)\to\R$ defined by $S(x)\coleq\widetilde{\alpha} x +\ln(1-x)$
\begin{linenomath*}
\begin{align}\label{S_asymp}
\begin{cases}
\textrm{is strictly decreasing } & \textrm{when } \quad  \widetilde{\alpha}\in (-\infty,1] \\[6pt]
\textrm{has a global maximum at }x=1-1/\widetilde{\alpha} \in(0,1) & \textrm{when } \quad  \widetilde{\alpha}\in (1,\infty) 
\end{cases}
\end{align}
\end{linenomath*}
Part (iii) then follows from  asymptotic expansion of integrals such as the Laplace method. 

Let $x^*\in[0,1]$ be the global maximum of the function $S$. Then $x^*=0$ when  $\widetilde{\alpha}\in (-\infty,1]$ and $x^*=1-1/\widetilde{\alpha}$ when $\widetilde{\alpha}\in (1,\infty)$. Fix an arbitrary $\epsilon\in(0,1)$. There exists $\delta\in(0,1)$ small enough such that $\sup_{y\in[0,1]:\,|y-x^*|<\delta}|f(y)-f(x^*)|<\epsilon$. For each of the two cases, by \eqref{S_asymp}, the ratio
\begin{linenomath*}
\begin{align}\label{ratioto0}
\int_{x\in[0,1]:\,|x-x^*|>\delta} e^{n_2\,S(x)}\,
x^{a-1}(1-x)^{\theta_2-1} e^{cx}\,dx \Big/
\int_{0}^1 e^{n_2\,S(x)}\,
x^{a-1}(1-x)^{\theta_2-1} e^{cx}\,dx \to 0 
\end{align}
\end{linenomath*}
as $n_2\to \infty$.
For any $f\in C_b([0,1])$, 
\begin{linenomath*}
\begin{align*}
&\Big|\int_0^1 f(x)\phi^{(n_1,n_2)}_{\alpha}(x)\,dx - f(x^*)\Big| \\
\leq & \Big|\int_{x\in[0,1]:\,|x-x^*|>\delta} f(x)\phi^{(n_1,n_2)}_{\alpha}(x)\,dx \Big| +  \Big|\int_{x\in[0,1]:\,|x-x^*|\leq \delta} f(x)\phi^{(n_1,n_2)}_{\alpha}(x)\,dx - f(x^*)\Big| \\
\leq& \|f\|\,\int_{x\in[0,1]:\,|x-x^*|>\delta} \phi^{(n_1,n_2)}_{\alpha}(x)\,dx + \epsilon +  |f(x^*)|\,\int_{x\in[0,1]:\,|x-x^*|>\delta}\phi^{(n_1,n_2)}_{\alpha}(x)\,dx.
\end{align*}
\end{linenomath*}
Hence by \eqref{ratioto0},
$\limsup_{n_2\to\infty}\Big|\int_0^1 f(x)\phi^{(n_1,n_2)}_{\alpha}(x)\,dx - f(x^*)\Big| \leq \epsilon$.
Since $\epsilon>0$ is arbitrary, we have shown that $\Big|\int_0^1 f(x)\phi^{(n_1,n_2)}_{\alpha}(x)\,dx - f(x^*)\Big|\to 0$ as $n_2\to\infty$. 
\end{proof}

\begin{proof}[Proof of Lemma \ref{L:Asymp_Initial}]
Convergence in distribution to a constant is equivalent to convergence in probability. Hence  Lemma \ref{L:Asymp_Initial}, except the last statement about the convergence in distribution of $n_2p_0$, follows from Lemma \ref{L:Convphi}. As in the main text, $A\approx B$ below means $A/B\to 1$ in the specified limit.

When $\widetilde{\alpha}\in  (-\infty,1)$, we  let $a=n_1+\theta_1$ for simplicity. The probability density function of $np_0$ under $\P_{\bf n}$ is
\begin{linenomath*}
\begin{align*}
\frac{1}{n}\phi^{(n_1,n_2)}_{\alpha}\left(\frac{y}{n}\right)= &\,
\frac{1}{n}\,\frac{1}{{\rm Beta}(a,n_2+\theta_2) {}_1F_1\left(a;n+\theta_1+\theta_2;\alpha\right)}\,\left(\frac{y}{n}\right)^{a - 1} \left(1-\left(\frac{y}{n}\right)\right)^{n_2+\theta_2 - 1} e^{\alpha \frac{y}{n}}\\
\approx &\,
\frac{1}{n}\,\frac{1}{{\rm Beta}(a,n_2+\theta_2) {}_1F_1\left(a;n+\theta_1+\theta_2;\,\widetilde{\alpha}n_2\right)}\,\left(\frac{y}{n}\right)^{a - 1} \,e^{-y}\,e^{\widetilde{\alpha} y}\\
\approx &\,
\frac{1}{n^a}\,\frac{1}{{\rm Beta}(a,n_2+\theta_2)\; (1-\widetilde{\alpha})^{-a}}\,y^{a - 1} \,e^{-y}\,e^{\widetilde{\alpha} y}\\
\approx &\,
\frac{1}{\Gamma(a)\; (1-\widetilde{\alpha})^{-a}}\,y^{a - 1} \,e^{-y}\,e^{\widetilde{\alpha} y}
\end{align*}
\end{linenomath*}
as $n_2\to\infty$,
where we used \eqref{eq:1F1ws1} and then \eqref{eq:gammaratio} in the last two approximations above.
Hence the probability density function of $n_2p_0$ (under $\P_{\bf n}$) converges pointwise to that of the ${\rm Gam}(n_1+\theta_1,1-\widetilde{\alpha})$ random variable. This implies the desired convergence in distribution.
\end{proof}

\begin{proof}[Proof of Lemma \ref{L:MoranVary}]
Fix $t\in\R_+$ and let $k=[N(N-1)t/2]$. Suppose $A(k)$ is the number of type 1 at step $k$ of the discrete-time Moran process. Direct calculations from \eqref{eq:pjplusone} and \eqref{eq:pjminusone} show that, as $N\to\infty$,
\[
\E\left[ \frac{A(k+1)-A(k)}{\rho(t)N} \,\Big|\,\frac{A(k)}{\rho(t)N}=x \right] \approx b(t,x)\frac{2}{N^2}
\]
and 
\[
\E\left[ \left(\frac{A(k+1)-A(k)}{\rho(t)N}\right)^2 \,\Big|\,\frac{A(k)}{\rho(t)N}=x \right] \approx \sigma^2(t,x)\frac{2}{N^2},
\]
where $b(t,x)\coleq \frac{\theta_1}{2 \rho(t)} (1-x) - \frac{\theta_2}{2 \rho(t)}x + \frac{\alpha}{2 \rho(t)} x (1-x)$  and $\sigma(t,x)=\sqrt{\frac{x(1-x)}{\rho^2(t)}}$ are the coefficients in \eqref{eq:sde_vary}.
The condition on $\rho$ guarantees that the SDE  \eqref{eq:sde_vary} has a unique weak solution and that the desired weak convergence follows from standard (martingale problem) method; for reference see \citet[Chapter 11]{stroock1979multidimensional}.
\end{proof}

\section{Events in the conditional ancestral selection graph} \label{sec:asgdetails}

Here we show how the minimal conditional ancestral process in Section~\ref{sec:condasg} is obtained from the full conditional ancestral process.  To begin, we assume that at some time in the conditional ancestral process there were $r_1$, $r_2$, $v_1$ and $v_2$ real and virtual lineages of type $1$ and type $2$.  The associated sampling probability is $q_o(r_1,r_2,v_1,v_2)$, the straightforward extension of \eqref{eq:qr1r2v1} or \eqref{eq:qr1r2v2} to include both type-$1$ and type-$2$ virtual lineages.  How branching events are resolved depends on which allele is favored by selection.  We begin here by assuming that $A_2$ is favored, or $\alpha<0$.  Grouping events by the types of lineages involved (real or virtual of type $1$ or type $2$) then by whether it is mutation, branching or coalescence gives fourteen possibilities which occur at the following rates.    
\begin{linenomath*}
\begin{align}
& r_1 \frac{\theta\pi_1}{2} {\left( \frac{q_o(r_1,r_2,v_1,v_2)}{q_o(r_1,r_2,v_1,v_2)} + \frac{q_o(r_1-1,r_2+1,v_1,v_2)}{q_o(r_1,r_2,v_1,v_2)} \right)} \label{eq:asgcases1} \\[6pt]
& r_2 \frac{\theta\pi_2}{2} {\left( \frac{q_o(r_1,r_2,v_1,v_2)}{q_o(r_1,r_2,v_1,v_2)} + \frac{q_o(r_1+1,r_2-1,v_1,v_2)}{q_o(r_1,r_2,v_1,v_2)} \right)} \label{eq:asgcases2} \\[6pt]
& v_1 \frac{\theta\pi_1}{2} {\left( \frac{q_o(r_1,r_2,v_1,v_2)}{q_o(r_1,r_2,v_1,v_2)} + \frac{q_o(r_1,r_2,v_1-1,v_2+1)}{q_o(r_1,r_2,v_1,v_2)} \right)} \label{eq:asgcases3} \\[6pt]
& v_2 \frac{\theta\pi_2}{2} {\left( \frac{q_o(r_1,r_2,v_1,v_2)}{q_o(r_1,r_2,v_1,v_2)} + \frac{q_o(r_1,r_2,v_1+1,v_2-1)}{q_o(r_1,r_2,v_1,v_2)} \right)} \label{eq:asgcases4} \\[6pt]
& r_1 \frac{\lvert\alpha\rvert}{2} \frac{q_o(r_1,r_2,v_1+1,v_2)}{q_o(r_1,r_2,v_1,v_2)} \label{eq:asgcases5} \\[6pt]
& r_2 \frac{\lvert\alpha\rvert}{2} {\left( 2 \frac{q_o(r_1,r_2,v_1+1,v_2)}{q_o(r_1,r_2,v_1,v_2)} + \frac{q_o(r_1,r_2,v_1,v_2+1)}{q_o(r_1,r_2,v_1,v_2)} \right)} \label{eq:asgcases6} \\[6pt]
& v_1 \frac{\lvert\alpha\rvert}{2} \frac{q_o(r_1,r_2,v_1+1,v_2)}{q_o(r_1,r_2,v_1,v_2)} \label{eq:asgcases7} \\[6pt]
& v_2 \frac{\lvert\alpha\rvert}{2} {\left( 2 \frac{q_o(r_1,r_2,v_1+1,v_2)}{q_o(r_1,r_2,v_1,v_2)} + \frac{q_o(r_1,r_2,v_1,v_2+1)}{q_o(r_1,r_2,v_1,v_2)} \right)} \label{eq:asgcases8} \\[6pt]
& \binom{r_1}{2} \frac{q_o(r_1-1,r_2,v_1,v_2)}{q_o(r_1,r_2,v_1,v_2)} \label{eq:asgcases9} \\[6pt]
& \binom{r_2}{2} \frac{q_o(r_1,r_2-1,v_1,v_2)}{q_o(r_1,r_2,v_1,v_2)} \label{eq:asgcases10} \\[6pt]
& r_1 v_1 \frac{q_o(r_1,r_2,v_1-1,v_2)}{q_o(r_1,r_2,v_1,v_2)} \label{eq:asgcases11} \\[6pt]
& r_2 v_2 \frac{q_o(r_1,r_2,v_1,v_2-1)}{q_o(r_1,r_2,v_1,v_2)} \label{eq:asgcases12} \\[6pt]
& \binom{v_1}{2} \frac{q_o(r_1,r_2,v_1-1,v_2)}{q_o(r_1,r_2,v_1,v_2)} \label{eq:asgcases13} \\[6pt]
& \binom{v_2}{2} \frac{q_o(r_1,r_2,v_1,v_2-1)}{q_o(r_1,r_2,v_1,v_2)} \label{eq:asgcases14} 
\end{align}
\end{linenomath*}
The sum of \eqref{eq:asgcases1} through \eqref{eq:asgcases14} is equal to the total rate of events in the unconditional ancestral process, $(r_1+r_2+v_1+v_2)(\theta+\lvert\alpha\rvert+r_1+r_2+v_1+v_2-1)/2$.  Twenty-two distinct events $(r_1,r_2,v_1,v_2)\to(r_1^\prime,r_2^\prime,v_1^\prime,v_2^\prime)$ are represented, one for each of the ratios of sampling probabilities, $q_0(r_1^\prime,r_2^\prime,v_1^\prime,v_2^\prime)/q_0(r_1,r_2,v_1,v_2)$.  Note that the assumption of parent-independent mutation leads to the four kinds of spurious or empty mutation events in \eqref{eq:asgcases1} through \eqref{eq:asgcases4} which do not change the ancestral state of the sample $(r_1^\prime=r_1,r_2^\prime=r_2,v_1^\prime=v_1,v_2^\prime=v_2)$.  Also, only those events which have have non-zero probabilities of giving the data appear in \eqref{eq:asgcases1} through \eqref{eq:asgcases14}; coalescent events between lineages with different types and type-$i$ mutation events on type $3-i$ lineages would make the data impossible.

Recall that the resolution of branching events depends on which allele is favored by selection.  The events and their probabilities in \eqref{eq:asgcases5} through \eqref{eq:asgcases8} are just for the case $\alpha<0$, where $A_2$ is the favored allele.  Each branching event creates an incoming lineage and a continuing lineage, both of which may be of type $1$ or type $2$.  Let $(I,C)$ be the types of these lineages.  
In \eqref{eq:asgcases5} and \eqref{eq:asgcases7}, only one of the four $(I,C)$ pairs has non-zero probability of producing the data: $(I=1,C=1)$ corresponding to the event $(r_1,r_2,v_1,v_2)\to(r_1,r_2,v_1+1,v_2)$.  In \eqref{eq:asgcases6} and \eqref{eq:asgcases8}, the possibility $(I=1,C=1)$ is discarded as it would then be impossible for the descendant lineage to be of type $2$.  The other three possibilities have non-zero chances of producing the data, and associated events 
\begin{linenomath*}
\begin{subnumcases}{(r_1,r_2,v_1,v_2) \to }
(r_1,r_2,v_1+1,v_2) & when $\,(I=1,C=2)$ \label{eq:I1C2neg} \\[2pt]
(r_1,r_2,v_1+1,v_2) & when $\,(I=2,C=1)$ \label{eq:I2C1neg} \\[2pt]
(r_1,r_2,v_1,v_2+1) & when $\,(I=2,C=2)$ \label{eq:I2C2neg} .
\end{subnumcases} 
\end{linenomath*}

In contrast, if $\alpha>0$ then branching events on type-$2$ lineages are the ones for which only one of the four $(I,C)$ pairs has non-zero probability of producing the data: $(I=2,C=2)$ corresponding to the event $(r_1,r_2,v_1,v_2)\to(r_1,r_2,v_1,v_2+1)$.  When $\alpha>0$, if the branching event occurs on a type-$1$ lineage, then in place of \eqref{eq:I1C2neg}, \eqref{eq:I2C1neg} and \eqref{eq:I2C2neg} we have 
\begin{linenomath*}
\begin{subnumcases}{(r_1,r_2,v_1,v_2) \to }
(r_1,r_2,v_1+1,v_2) & when $\,(I=1,C=1)$ \label{eq:I1C1pos} \\[2pt]
(r_1,r_2,v_1,v_2+1) & when $\,(I=1,C=2)$ \label{eq:I1C2pos} \\[2pt]
(r_1,r_2,v_1,v_2+1) & when $\,(I=2,C=1)$ \label{eq:I2C1pos} .
\end{subnumcases} 
\end{linenomath*}
Therefore, when $\alpha>0$, \eqref{eq:asgcases5} through \eqref{eq:asgcases8} must be replaced with 
\begin{linenomath*}
\begin{align}
& r_1 \frac{\alpha}{2} {\left( \frac{q_o(r_1,r_2,v_1+1,v_2)}{q_o(r_1,r_2,v_1,v_2)} + 2 \frac{q_o(r_1,r_2,v_1,v_2+1)}{q_o(r_1,r_2,v_1,v_2)} \right)} \label{eq:asgcases15} \\[6pt]
& r_2 \frac{\alpha}{2} \frac{q_o(r_1,r_2,v_1,v_2+1)}{q_o(r_1,r_2,v_1,v_2)} \label{eq:asgcases16} \\[6pt]
& v_1 \frac{\alpha}{2} {\left( \frac{q_o(r_1,r_2,v_1+1,v_2)}{q_o(r_1,r_2,v_1,v_2)} + 2 \frac{q_o(r_1,r_2,v_1,v_2+1)}{q_o(r_1,r_2,v_1,v_2)} \right)} \label{eq:asgcases17} \\[6pt]
& v_2 \frac{\alpha}{2} \frac{q_o(r_1,r_2,v_1,v_2+1)}{q_o(r_1,r_2,v_1,v_2)} \label{eq:asgcases18} 
\end{align}
\end{linenomath*}
Equations \eqref{eq:asgcases1} through \eqref{eq:asgcases7} and \eqref{eq:asgcases9} through \eqref{eq:asgcases14} are the same for $\alpha>0$ and $\alpha<0$.  

The simplifications discovered by \citet{Slade2000a} and \citet{Fearnhead2002} follow from the simple fact that each sampled lineage is either of type $1$ or type $2$.  \citet{Slade2000a} noticed that when both the descendant lineage and the incoming lineage have the favored type, the type of the continuing lineage does not matter so there is no need to introduce a new virtual lineage. Instead, these two possibilities can be collapsed into a single null event which does not change the numbers and types of ancestral lineages.  That is, we can use  
\begin{linenomath*}
\begin{equation}
q_o(r_1,r_2,v_1+1,v_2) + q_o(r_1,r_2,v_1,v_2+1) = q_o(r_1,r_2,v_1,v_2)
\end{equation}
\end{linenomath*}
in \eqref{eq:asgcases6}, \eqref{eq:asgcases8}, \eqref{eq:asgcases15} and \eqref{eq:asgcases17}.  As a result, no type-$2$ virtual lineages will be created.  

Along the same lines, \citet{Fearnhead2002} noticed that when mutation is parent-independent there is no need to follow ancestral lineages once they have mutated, because the ancestral lineage could be of either type.  Any such lineage can be removed from the ancestral process.  Here we use 
\begin{linenomath*}
\begin{equation}
q_o(r_1,r_2,v_1,v_2) + q_o(r_1+1,r_2-1,v_1,v_2) = q_o(r_1,r_2-1,v_1,v_2) 
\end{equation}
\end{linenomath*}
in \eqref{eq:asgcases2}, and other appropriate identities in \eqref{eq:asgcases3} and \eqref{eq:asgcases4}.  But we do not make use of this simplification in \eqref{eq:asgcases1} because our specific goal is to model latent mutations in the ancestry of $A_1$.  These are actual mutations, where the ancestral type was $A_2$.  The remaining $A_1 \to A_1$ empty mutations are null events, which do not change the numbers and types of ancestral lineages.   

The conditional ancestral processes for $\alpha<0$ and $\alpha>0$ given by \eqref{eq:asgn1alphanegcases} and \eqref{eq:asgn1alphaposcases} in the main text each include just five kinds of (non-null) events.  We obtain these by applying the simplifications of \citet{Slade2000a} and \citet{Fearnhead2002} then grouping events by their outcomes.  For example, the coalescent events in \eqref{eq:asgcases10} have effect $r_2 \to r_2-1$, as do the combined mutations in \eqref{eq:asgcases2} once the simplification of \citet{Fearnhead2002} is applied.  So these appear together as one kind of event, the fifth case in both \eqref{eq:asgn1alphanegcases} and \eqref{eq:asgn1alphaposcases}.  

We do not include null events in \eqref{eq:asgn1alphanegcases} and \eqref{eq:asgn1alphaposcases} since these by definition have no effect on the ancestral lineages.  In the case $\alpha<0$, the null events are empty mutations on type-$1$ real lineages and branching events on type-$2$ real lineages where the incoming line is also of type $2$.  These occur with total rate $r_1\theta\pi_1/2 + r_2\lvert\alpha\rvert/2$.  In the case $\alpha>0$, the null events are empty mutations on type-$1$ real lineages and branching events on type-$1$ real lineages where the incoming line is also of type $1$.  These occur with total rate $r_1\theta\pi_1/2 + r_1\alpha/2$.  


%\bibliographystyle{elsarticle-harv} 
%\bibliography{refs}


%\begin{thebibliography}{00}



%\end{thebibliography}

% cut-and-pasted contents of bbl file generated using elsarticle-harv as above

\end{document}